\documentclass[11pt, reqno]{amsart}

\usepackage[utf8]{inputenc}
\usepackage[english]{babel}
\usepackage[T1]{fontenc}

\title[Median filter for MCF]{Median filter method for mean curvature flow using a random Jacobi algorithm}

\author{Tim Laux}
\thanks{T.L.\ Faculty of Mathematics, University of Regensburg, Universitätsstrasse 31, 93053
Regensburg, Germany. \nolinkurl{tim.laux@ur.de}}

\author{Anton Ullrich}
\thanks{A.U.\ Max Planck Institute for Mathematics in the Sciences, Inselstraße 22, 
04103 Leipzig, Germany. \nolinkurl{anton.ullrich@mis.mpg.de}}

\date{}

\usepackage{stmaryrd}
\usepackage{enumitem}

\usepackage{algorithm2e}
\SetKwComment{Comment}{/* }{ */}
\RestyleAlgo{ruled}
\usepackage{listings}
\usepackage{verbatim}
\usepackage{stmaryrd}
\usepackage{dsfont}
\usepackage{tikz}
\usetikzlibrary{calc, quotes, angles}
\usepackage{float}
\usepackage{amsmath}
\usepackage{amssymb}
\usepackage{amsthm}
\usepackage{mathrsfs}

\usepackage{mathtools}
\usepackage{esint}
\usetikzlibrary{decorations.pathreplacing}

\usepackage{algorithm2e}

\PassOptionsToPackage{dvipsnames}{xcolor}
\RequirePackage{xcolor}
\definecolor{webgreen}{rgb}{0,.5,0}
\definecolor{webbrown}{rgb}{.6,0,0}
\definecolor{RoyalBlue}{cmyk}{1, 0.50, 0, 0}

\usepackage[colorlinks,
pdfpagelabels,
pdfstartview = FitH,
bookmarksopen = true,
bookmarksnumbered = true,
linkcolor = RoyalBlue,
plainpages = false,
urlcolor=webbrown,
hypertexnames = false,
citecolor = webgreen] {hyperref}

\newcommand*{\dd}{\mathop{}\!\mathrm{d}}
\newcommand{\vect}[1]{\begin{pmatrix}#1\end{pmatrix}}
\newcommand{\norm}[1]{\left\lVert#1\right\rVert}

\DeclareMathOperator*{\sign}{sign}

\DeclareMathOperator*{\Lip}{Lip}
\DeclareMathOperator*{\BUC}{BUC}
\DeclareMathOperator*{\TL}{TL}
\DeclareMathOperator*{\TV}{TV}
\DeclareMathOperator*{\argmin}{argmin}

\DeclareMathOperator*{\spn}{span}
\DeclareMathOperator*{\dist}{dist}
\DeclareMathOperator*{\ddiv}{div}
\DeclareMathOperator*{\vol}{vol}
\DeclareMathOperator*{\Per}{Per}
\DeclareMathOperator*{\med}{med}
\DeclareMathOperator*{\pmed}{med_p}
\DeclareMathOperator*{\ulim}{{\lim}^*}
\DeclareMathOperator*{\llim}{{\lim}_*}

\theoremstyle{plain}
\newtheorem{defi}{Definition}[section]

\newtheorem*{thm*}{Theorem}
\newtheorem*{ex*}{Example}
\newtheorem*{defi*}{Definition}
\newtheorem*{remark*}{Remark}
\newtheorem*{prop*}{Proposition}
\newtheorem{thm}[defi]{Theorem}
\newtheorem{Corollary}[defi]{Corollary}
\newtheorem{prop}[defi]{Proposition}
\newtheorem{lemma}[defi]{Lemma}
\newtheorem{remark}{Remark}
\newtheorem*{lemma*}{Lemma}

\usepackage{pgfplots}
\pgfplotsset{compat=newest}
\usetikzlibrary{intersections, pgfplots.fillbetween}
\pgfdeclarelayer{pre main}
\usepackage{makecell}
\allowdisplaybreaks

\usepackage{aligned-overset}
\usepackage{tabularray}

\begin{document}

\begin{abstract}
We present an efficient scheme for level set mean curvature flow using a domain discretization and median filters. For this scheme, we show convergence in $L^\infty$-norm under mild assumptions on the number of points in the discretization. In addition, we strengthen the weak convergence result for the MBO thresholding scheme applied to data clustering of Lelmi and one of the authors. This is done through a strong convergence of the discretized heat flow in the optimal regime. Different boundary conditions are also discussed.
\end{abstract}

\keywords{Mean curvature flow, viscosity solution, thresholding scheme, median filter, Young angle condition, continuum limit.
\emph{MSC2020:} 53E10 (Primary); 35D40; 60D05 (Secondary)}

\maketitle

\tableofcontents

\section{Introduction}
\label{Introduction}

This paper presents a fast algorithm for the approximation of level set mean curvature flow and proves its convergence. At its core, the algorithm is a median filter based evolution and is thus equivalent to the MBO scheme on any level set.
The novelty of this algorithm is that it works simultaneously on all level sets on discretized domains in a way that easily generalizes to high dimensions. It is similar to the median filter algorithm proposed by \cite{esedoglumedianfilter} which in turn is based on the algorithm of \cite{Oberman}.
For our evolution, we show almost sure convergence in the strong $L^\infty$-topology under the assumption that the discretization has sufficiently many points, see Theorem~\ref{Cor:Conv} for details. In addition, we obtain the convergence of a heat flow in a related discretization and thus strengthen the weak convergence results of the MBO scheme obtained by \cite{jonatim}, see Theorem~\ref{Thm:TLConv} and the following corollaries. Finally, we prove the $\Gamma$-convergence of the energies associated with our evolution even under inhomogeneous Young angle conditions at the boundary.

Mean curvature flow and the heat flow are among the most important and applied parabolic evolutions. The mean curvature flow finds applications in crystal growth, material property analysis, soap films but more recently also in machine learning, image denoising, classification and clustering. For these more modern applications, it is crucial to develop and understand efficient numerical methods in high dimensions.

For smooth submanifolds it can be described as the evolution in time that is given by proposing a normal velocity equal to the negative mean curvature. Formally, this can be seen as a gradient flow of the area functional which also rectifies the name curve shortening flow in two dimensions. Stationary objects of this flow are minimal surfaces, i.e., have vanishing mean curvature. This leads to the use of the evolution in classification problems.

In this paper, we are mainly interested in the level set mean curvature flow. This is a flow on functions where almost every level set evolves by mean curvature flow. The formulation of this flow is due to Osher and Sethian~\cite{OsherSethian}. Based on the level set flow, we will study the viscosity solution and convergences to this solution. The concept of viscosity solutions was introduced independently by Evans and Spruck~\cite{ESI}, and Chen, Giga and Goto~\cite{ChenGigaGoto}.

The MBO scheme is an efficient algorithm for the approximation and simulation of mean curvature flow that was proposed by Merriman, Bence and Osher in \cite{merriman1992diffusion}. It is based on a diffusion and thresholding coupling that is iterated. For a detailed explanation, see Section~\ref{Algorithm}.

\subsection{Structure of the paper}

First, in Section~\ref{Algorithm}, we present the setting of the scheme and give the notation needed to formulate our main theorem for the convergence to viscosity solutions of mean curvature flow, Theorem~\ref{Cor:Conv}. Along with this, we introduce the definition of mean curvature flow, notions of weak solution concepts and various convergences. In addition, we discuss the algorithm and related schemes.
Roughly speaking, the algorithm is based on the observation that the median corresponds to the level set Laplacian in the same way that the mean corresponds to the usual Laplacian. Together with the fact that level set mean curvature flow is described by the parabolic equation of this operator, we get an efficient and simple algorithm to approximate this flow. More precisely, the algorithm takes a discretization of our domain with a given initial datum of the flow. Then, in each step, the value of each point is updated by taking the median in a local neighborhood. Different possible choices of neighborhoods and stencils are discussed with their advantages and disadvantages. This scheme can also be seen as a direct extension of the MBO scheme to level set functions.

Afterwards, in Section~\ref{LinfConv}, we prove that the algorithm converges uniformly to the viscosity solution. This is done in two parts. First, we prove that for the continuum limit (in space), the algorithm approximates level set mean curvature flow in terms of viscosity solutions uniformly in time and space, see Lemma~\ref{UniformConvergence}. The proof is based on the results of Ishii, Pires, Souganidis, see~\cite{IPS}. Next, we show that given enough points in the local neighborhood, the discrete median of randomly distributed points approximates the continuous median almost surely. Thus, our discrete scheme converges. This is shown in Theorem~\ref{Cor:Conv}, which can be summarized as follows:
\begin{thm*}[Version of Theorem~\ref{Cor:Conv}]
    For an initial datum $g$ there exists a regime of step size and number of samples (step size going to $0$ and sample size to $\infty$) s.t.\ the solutions of the scheme converge almost surely to the viscosity solution of mean curvature flow locally uniformly in time and space.
\end{thm*}
Both parts are done in $\mathbb{R}^d$, more precisely in the flat torus $\mathbb{T}^d$ without boundary. However, the authors believe that other generalizations similar to the other results in later sections are also possible. For example, we analyze the scheme for different annuli and changing kernels but other symmetric kernels would work as well. One could also change the density used in the measure to a smooth density. If the domain changes to a general compact Lipschitz domain in $\mathbb{R}^d$ or in a manifold, one would have to include boundary conditions. The natural condition for this would be homogeneous Neumann conditions, while even inhomogeneous can be done with the changes described in Section~\ref{Program}. On a manifold, one would have to include the curvature of the underlying structure. In addition, the sampling can be modified. We could also ask the question e.g.\ for other random point processes or lattices.

In Section~\ref{TL-convergence}, we show another notion of convergence for the heat flow that can be used for the MBO scheme. This approach gives better results for the number of points. We can even achieve the optimal regime, because the errors are averaged over time and the evolutions become calibrated to the discrete structure. This could not be captured by the pointwise approach of the previous section. Similar to the result of Section~\ref{LinfConv}, we show the convergence of a fully discretized mean scheme for the heat equation to the continuous heat flow, see Theorem~\ref{Thm:TLConv}. Thus, we conclude the same convergence for the heat flows in the sampled and continuous setting, see Theorem~\ref{Cor:HeatFlows}. The convergence can be roughly stated as follows
\begin{thm*}[Version of Theorem~\ref{Cor:HeatFlows}]
    Let initial data $g_N$ on the discretized domains be given in such a way that they convergence weakly to the initial datum $g$ for the heat flow $u$. Moreover, denote by $u_N$ the discrete heat flows starting at $g_N$. Then, we have almost surely that $u_N(t)$ converges strongly to $u(t)$ for every time $t>0$.
\end{thm*}
This convergence is the ${\TL}^2$-convergence which can be understood as a $L^2$ variant for discrete settings. As a consequence of this convergence, we strengthen the result of \cite{jonatim} with their technique to obtain a separated weak convergence of the MBO scheme. These results are obtained for more general kernels, densities and domains, even on manifolds which is summarized in Corollary~\ref{Cor:generalSettingTL}. Since the structure of the sampling only plays a role in the $\Gamma$-convergence of the energies, we can directly generalize this part to point processes with weak assumptions on their distribution. The results are obtained for homogeneous Neumann conditions but we discuss in Section~\ref{Program} that other conditions are feasible as well.

Next, in Section~\ref{Program}, we discuss the implementations and applications of our scheme. The different implementations can be found on \url{https://mathrepo.mis.mpg.de/Medianfilter/}. We also analyze variants of the scheme. In particular, we will show a generalization to Young angle conditions which extend the Neumann conditions. For these boundary conditions, we propose a new scheme and show the $\Gamma$-convergence of the associated energies, see Lemma~\ref{Lemma:NeumannConditions}.
In the following section, Section~\ref{History}, other papers and related work are discussed. Finally, in Section~\ref{Outlook}, we give an outlook on interesting open questions related to this work.

\subsection{Notation}

\begin{figure}[!ht]
    \centering
    \begin{tabular}{|c|l| } 
    \hline
    Symbol & Explanation \\
    \hline\hline
    $\mathbb{P}$ & \makecell[vl]{probability measure associated with the uniform\\ density on the domain}\\
    \hline
    $\mathbb{P}_N$ & \makecell[vl]{probability measure associated with point process\\ of sampled points in domain}\\
    \hline
    $P$ & Set of points of the (Poisson) point process\\
    \hline
    $\Lambda$ & intensity for the Poisson point process\\
    \hline
    $r$ & space variable and distance, equal to $\sqrt{h}$\\
    \hline
    $h$ & time step size, equal to $r^2$\\
    \hline
    $g$ & \makecell[vl]{initial data of scheme in $\BUC(D)$\\ (bounded uniformly continuous functions)}\\
    \hline
    $u^n_h$ & step-wise evolution\\
    \hline
    $u(x,t)$ & viscosity solution to level set mean curvature flow of $g$\\
    \hline
    $G_hg$ & one step of the scheme with time step size $h$ applied to $u$\\
    \hline
    $m_\mathbb{P}$ & continuous median in the neighborhood domain\\
    \hline
    $m_N$ & \makecell[vl]{discrete median based on the Point process\\ in the neighborhood domain}\\
    \hline
    $Q_{h,t}^{(\mathbb{P}/N)}g$ & piecewise constant evolution with $m_{\mathbb{P}/N}$\\
    \hline
    $T$ & final time for evolution\\
    \hline
    $\mathbb{T}^d$ & \makecell[vl]{the $d$-dimensional flat torus \\($[0,1]^d$ with periodic boundary data)}\\
    \hline
    $N$ & total amount of points in the domain $\mathbb{T}^d$\\
    \hline
    $A_r$ & neighborhood domain of typical size $r$\\
    \hline
    $A_{\kappa,r}$ & the annulus $B_r\setminus B_{\kappa\cdot r}$\\
    \hline
    $\omega_d$ & the volume of the unit ball $B_1$ in $\mathbb{R}^d$ \\
    \hline
    $\mathds{1}_A$ & indicator function for $A$, i.e., $\mathds{1}_A(x)=1$ if $x\in A$ and $0$ else \\
    \hline
    $\widetilde \cdot$ & normalized quantity \\
    \hline
    $E_{r,N}$ & non-local, discretized energy \\
    \hline
    $E$ & continuous energy \\
    \hline
    $\nu_A$ & outward unit normal vector to the domain $A$ \\
    \hline
    \end{tabular}
    \caption{The notation of this paper.}
    \label{table:notation}
\end{figure}

We will use the Landau notation to talk about vanishing quantities depending on a variable $r\to 0$.

\begin{defi}[Landau notation]
    For a sequence $r\to 0$, we say that 
    \begin{itemize}
        \item $f(r)\in o(g(r))$ if $\lim\limits_{r\to 0}\frac{f(r)}{g(r)}\to 0$,
        \item $f(r)\in \mathcal{O}(g(r))$ if $\limsup\limits_{r\to 0}\frac{f(r)}{g(r)}<\infty$,
        \item $f(r)\in \Omega(g(r))$ if $\liminf\limits_{r\to 0}\frac{f(r)}{g(r)}>0$.
    \end{itemize}
\end{defi}

Further notations can be found in Figure~\ref{table:notation}.

\section{Setting \& Algorithm}
\label{Algorithm}

In this paper, we will examine an algorithm that approximates the level set mean curvature flow in terms of viscosity solutions. For this purpose, consider the domain $D$ which we will assume for the moment to be $D\coloneqq \mathbb{T}^d$, the $d$-dimensional flat torus, and an initial datum $g:D\to\mathbb{R}$ lying in $\BUC(D)$, the bounded uniform continuous functions. We will always consider $d\geq 2$.
Various modifications are applicable and will be discussed at their respective places. For example, we will consider general Lipschitz domains $D\subset \mathbb{R}^d$ with Neumann conditions or smooth manifolds with a continuous density.
We want to show the convergence to the viscosity solution of level set mean curvature flow $u$ with initial datum $g$. The equation for this evolution is
\begin{align}
    \label{Eq:Levelset}
    \partial_t u=|\nabla u|\nabla\cdot \frac{\nabla u}{|\nabla u|}=\Delta u-\frac{\nabla u}{|\nabla u|}\cdot \nabla^2u \frac{\nabla u}{|\nabla u|}.
\end{align}

\begin{remark*}
    If $u$ is smooth with $\nabla u\neq 0$, this evolution implies that the level sets of $u$ evolve by mean curvature flow.
\end{remark*}
This equation also can be verified rigorously and intuitively. For a simple reasoning, consider a trajectory $x(t)\in D$ for $t\in (t_1,t_2)$ following one level set, i.e., $u(x(t),t) = \text{const.}$ and suppose $\nabla u(x(t),t) \neq 0$ for all $t\in(t_1,t_2)$. 
The outer unit normal vector to the sub level sets is given by $\nu(x,t)=\frac{\nabla u(x,t)}{|\nabla u(x,t)|}$ and since the mean curvature is the divergence of the outer unit normal, $H=\nabla\cdot \nu$, the mean curvature flow equation can be expressed by $\dot x\cdot \nu=-\nabla\cdot \nu$. Thus
\begin{align*}
0=\frac{\dd}{\dd t}u(x(t),t)
=\partial_t u+\dot x\cdot\nabla u
=\partial_t u-|\nabla u|\nabla\cdot\nu
\end{align*}
and hence, we obtain the level set equation for $u$:
$$\partial_t u=|\nabla u|\nabla\cdot \nu=\Delta u-\frac{\nabla u}{|\nabla u|}\cdot \nabla^2 u\frac{\nabla u}{|\nabla u|}.$$

For our algorithm, we fix a sample size $N\in\mathbb{N}$ and a localization radius $r>0$. This radius corresponds to a time step size $h\coloneqq r^2$ via parabolic scaling. Given a uniform sample $P$ of the domain with $N$ points in expectation (e.g., $N$ uniformly sampled points or a Poisson point process with intensity $N$), we construct the approximation $u_h^n(x)$ by a piecewise constant interpolation in time of the following iteration. Let $u_h^0(x)$ be a given initial datum $u_h^0\coloneqq g\in \BUC(D)$ and define iteratively
$$u_h^{n+1}(x)\coloneqq\med_{A_r(x)\cap P} u_h^n.$$
This means that in each step we take the simultaneous median $\med\limits_{A_r\cap P}$ in a local neighborhood $A_r$ over the sampled process for each point in the domain. A visualization of this algorithm can be seen in Figure~\ref{Fig:AlgVisualization}.

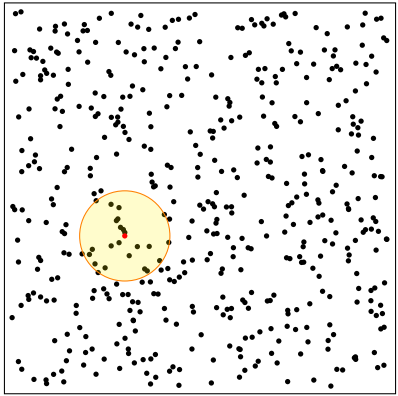
\begin{figure}[!ht]
\label{Fig:AlgVisualization}
\centering
\begin{tikzpicture}
\foreach \i in {1,...,500}
  \fill (rnd*5cm, rnd*5cm) circle (1pt);
\draw[draw=black] (-.1,-.1) rectangle (5.1,5.1);
\draw[draw=orange, fill=yellow, fill opacity=0.2] (1.5,2) circle (0.6);
\fill[red] (1.5,2) circle (1pt);
\end{tikzpicture}
\caption{Visualization of the algorithm.}
\end{figure}

We will consider different kernels. Specifically, $A_r$ will be either the ball $B_r$, an annulus $A_{\kappa, r}\coloneqq B_r\setminus B_{\kappa r}$ or shrinking annuli where $\kappa\to 1$ as $r\to 0$. 

\begin{algorithm}
\caption{The median filter scheme}\label{alg:Main}
\KwData{$g\in \BUC(D;\mathbb{R}), T\in\mathbb{R}_{>0}, P_N$}
\KwResult{Evolution $u^n_h$}
$u^0_h(x)\coloneqq g(x)\ \forall x\in P_N$\;
$n\gets 0$\;
 \While{$\frac{nh}{c_A}<T$}{
  $u^{n+1}_h(x)\coloneqq {\med}_N(u^n_h;A_r(x))\ \forall x\in P_N$\;
  $n\gets n+1$\;
 }
\end{algorithm}

In Algorithm~\ref{alg:Main}, we give a pseudo-code description of our scheme.
Here, ${\med}_N(u^n_h;A_r(x))\coloneqq \med\limits_{A_r(x)\cap P_N} u^n_h$ is the discrete median on the points of $P_N$ that lie in $A_r$ and the constant $c_A$ depends on the neighborhood domain $A_r$ as we show in Section~\ref{LinfConv}. The exact value of $c_A$ is explicit, see e.g., Remark~\ref{Rem:constantCA}.

In Section~\ref{DiscreteCase}, we can also use the estimate with a given number of points $N$ (instead of an intensity $N$) and assume a uniform i.i.d.\ sampling of $N$ points instead of a Poisson point process. This more closely models our program and yields the same estimates.

Next, we give an intuition why the median filter is the extension of the thresholding MBO scheme on characteristic functions to level set functions. For a set, or equivalently a characteristic function $\chi^n$, a step of the MBO scheme is given by a diffusion step with (normalized) kernel $K$ followed by a thresholding step:
$$\chi^{n+1}\coloneqq \mathds{1}_{\{K*\chi^n\geq \frac{1}{2}\}}.$$
We note that this is the solution to the minimization problem
\begin{align*}
    \chi^{n+1}&\in\argmin\left\{\,\int\limits_{\mathbb{R}^d}\chi(1-2K*\chi^n)\dd x\right\}\\
    &=\argmin\left\{\,\int\limits_{\mathbb{R}^d}\int\limits_{\mathbb{R}^d}K(x,y)|\chi(x)-\chi^n(y)|\dd y\dd x\right\}.
\end{align*}
If $\chi^n=\mathds{1}_{\{u^n<q\}}$ is the characteristic function of the sub level set of a function~$u^n$ and we integrate this minimization problem over all level sets, we arrive at
\begin{align*}
    u^{n+1}\in\argmin\left\{\,\int\limits_{\mathbb{R}^d}\int\limits_{\mathbb{R}^d}K(x,y)|u(x)-u^n(y)|\dd y\dd x\right\}.
\end{align*}
The Euler-Lagrange equation to this problem is
\begin{align*}
    0&=\int\limits_{\mathbb{R}^d}K(x,y)\sign(u(x)-u^n(y))\dd y.
\end{align*}
A solution of this equation is the median associated with the kernel $K$. This becomes the usual median as described in Section~\ref{LinfConv} for a kernel that is a (normalized) characteristic function.
More precisely, it may happen that the solution to this equation and the minimizers are not unique. In this case, we call the entire set of all minimizers the median and will usually consider the infinimum of it (which corresponds to maximal $\chi$ as done above). Additionally, for functions that have fattening, i.e., level sets with a positive mass, it can happen that the equation has no solution. In this situation, we have to consider a solution of the inequalities
$$\frac{1}{2}\leq \int\limits_{\mathbb{R}^d}K(x,y)\mathds{1}_{\{u(x)\leq u^n(y)\}}\dd y$$
and 
$$\frac{1}{2}\leq \int\limits_{\mathbb{R}^d}K(x,y)\mathds{1}_{\{u(x)\geq u^n(y)\}}\dd y,$$
see also the beginning of Section~\ref{LinfConv}.

Figure~\ref{fig:EvolEllipse} shows an example of the evolution of the proposed algorithm. The white circle in the upper left corner shows the size of the kernel $A_r$. The colored bands indicate some level sets of the function $u$ and the red line is a comparison of the evolution starting from a level set computed by a direct Euler method with a very small time step.

\begin{figure}[!ht]
    \centering
    \includegraphics[width=0.3\linewidth]{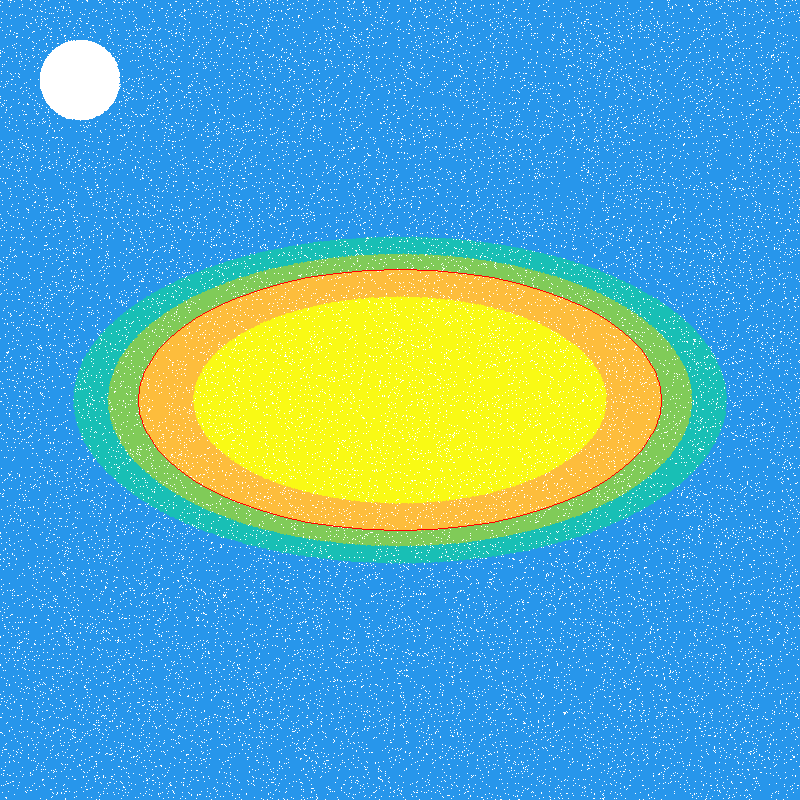}
    \includegraphics[width=0.3\linewidth]{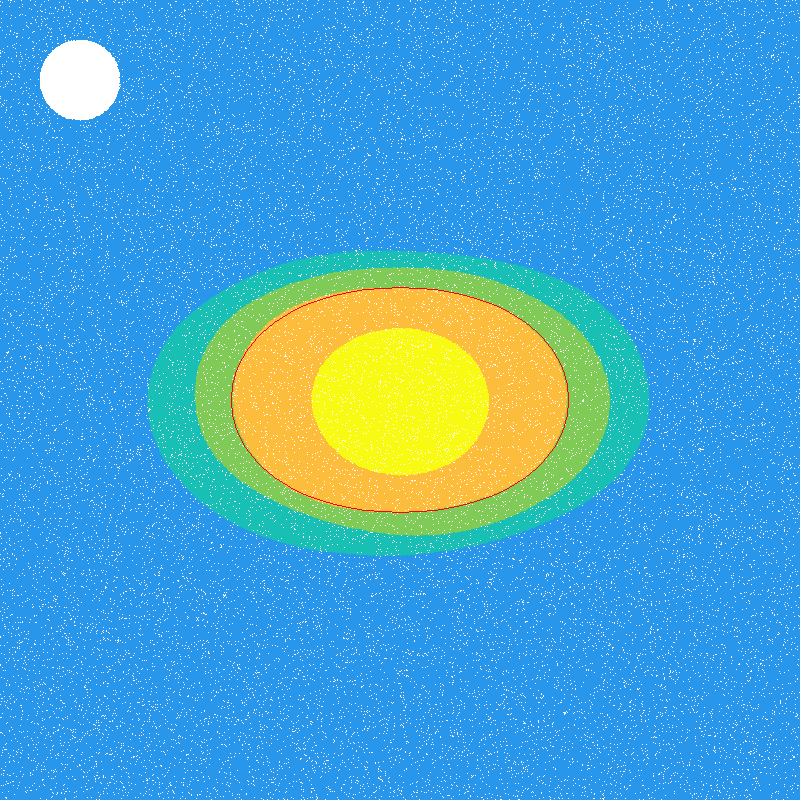}
    \includegraphics[width=0.3\linewidth]{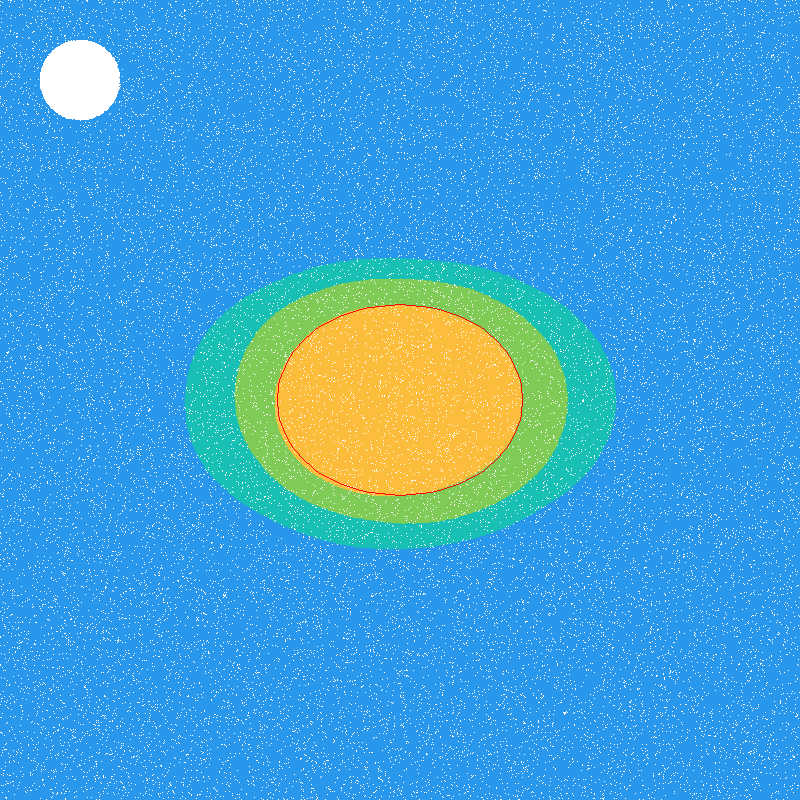}
    \caption{The evolution of a  non-symmetric function.}
    \label{fig:EvolEllipse}
\end{figure}

\subsection{Definitions}

In Section~\ref{TL-convergence}, we will compare evolutions on a discretized space with their continuous counterparts in $L^2$. To do this, we will need the notion of ${\TL}^2$-distances. This is due to \cite{TrillosSlepcev}.

\begin{defi}[${\TL}^2$-space]
    \label{Def:TL2}
    We define the space of compatible measure-function pairs $(\mu,f)$ as 
    $${\TL}^2(D)\coloneqq \{(\mu,f) \left| \mu\in \mathcal{P}(D), f\in L^2(D,\mu) \right\}.$$
\end{defi}

\begin{defi}[${\TL}^2$-convergence]
    \label{Def:TL2Conv}
    We define for $(\mu,f),(\nu,g)\in {\TL}^2(D)$ the distance $d_{{\TL}^2}$ as
    $$d_{{\TL}^2}((\mu,f),(\nu,g))\coloneqq \inf\limits_{\pi\in \Gamma(\mu,\nu)}\left(\int\limits_D\int\limits_D |x-y|^2+|f(x)-g(y)|^2\dd\pi(x,y)\right)^\frac{1}{2}$$
    and say that a sequence $\{(\mu_n,f_n)\}_{n\in\mathbb{N}}$ converges in ${\TL}^2$ if
    $$d_{{\TL}^2}((\mu_n,f_n),(\mu,f))\to 0.$$
    Here, $\Gamma(\mu,\nu)$ is the set of all transport plans from $\mu$ to $\nu$.
\end{defi}

This distance is indeed a metric on ${\TL}^2(D)$. Intuitively, it measures how to bring the mass from a function living on one measure space to the mass of another function on a different measure space. This will be the right notion to compare a function on the continuous space to a function on the sampled discrete space.

By an abuse of notation, we will sometimes write $u_N\to u$ in ${\TL}^2(D)$ for $u_N\in L^2(D,\mu_N), u\in L^2(D,\mu)$ instead of $(\mu_N, u_N)\to (\mu, u)$ when it is clear which measures are meant.

To make it clear what the initial datum is and what kind of evolution (fully discretized, time continuous and space discretized, non-local or continuous) we mean, we will use the notation $G_hg$ for one step of the algorithm and 
\begin{align}
    Q_{h,t}g\coloneqq G_h^jg \text{ for }t\in[jh,(j+1)h)
    \label{Eq:DefQ}
\end{align}
with $j\in\mathbb{N}_{\geq 0}$ for the entire piecewise constant-in-time evolution. Additional indices and exponents are used to specify the concrete setting and will be explained at their respective places. For the time continuous evolution, we will use $u(t,x)$ possibly with an index to specify a space discretization.

The continuous limit of our algorithm will be the viscosity solution to level set mean curvature flow.

\begin{defi}[Viscosity solution] \label{def:viscositysolution}
A continuous function $u\in C(D\times[0,T])$ is called a viscosity super-solution to \eqref{Eq:Levelset} if, for any $\varphi\in C^\infty(D\times[0,T])$ and $(x_0,t_0)\in D\times(0,T)$ such that $u-\varphi$ has a local minimum at $(x_0,t_0)$, the following inequality holds,
\begin{enumerate}[label=\roman*)]
    \item for non-critical points $\nabla \varphi(x_0,t_0)\neq 0:$
    $$\partial_t \varphi\geq \Delta \varphi-\frac{\nabla \varphi}{|\nabla \varphi|}\cdot \nabla^2\varphi\frac{\nabla \varphi}{|\nabla \varphi|}\quad\text{at }(x_0,t_0)$$
    \item and for critical points $\nabla \varphi(x_0,t_0)=0$ there exists $\xi\in\mathbb{R}^d, |\xi|\leq 1:$
    $$\partial_t \varphi\geq \Delta \varphi-\xi\cdot \nabla^2\varphi \xi\quad\text{at }(x_0,t_0).$$
\end{enumerate}
We say that $u$ is a viscosity sub-solution if $-u$ is a viscosity super-solution. Finally, $u$ is called a viscosity solution if it is both a viscosity sub- as well as a viscosity super-solution.
\end{defi}

\begin{defi}[Level set Laplacian]
\label{Def:F}
    We define the operator for viscosity solutions to level set mean curvature flow, the level set Laplacian, via
    $$F(\nabla \varphi, \nabla^2\varphi)\coloneqq\begin{cases}
    \Delta \varphi-\frac{\nabla \varphi}{|\nabla \varphi|}\cdot \nabla^2\varphi \frac{\nabla \varphi}{|\nabla \varphi|} &\text{for }\nabla\varphi\neq 0,\\ [\Delta \varphi-\lambda_d,\Delta \varphi-\lambda_1] &\text{for }\nabla\varphi= 0.
    \end{cases}$$
    Here, $\lambda_1\leq \dots\leq \lambda_d$ denote the eigenvalues of $\nabla^2 \varphi$.
\end{defi}
In the degenerate case, this is a set-valued operator. The set is the interval spanning from the sum of the smallest $d-1$ eigenvalues of $\nabla^2\varphi$ to the sum of the largest $d-1$ eigenvalues.

\begin{remark}
    With the definition of $F$, one can rewrite the conditions in Definition~\ref{def:viscositysolution} as
    $$\partial_t \varphi \geq \inf F(\nabla \varphi, \nabla^2\varphi)(x_0,t_0)$$
    for super-solutions and 
    $$\partial_t \varphi \leq \sup F(\nabla \varphi, \nabla^2\varphi)(x_0,t_0)$$
    for sub-solutions.

    In the case that $u$ is smooth, this condition becomes $\partial_t u\in F(\nabla u, \nabla^2 u)$. 
\end{remark}

\begin{remark*}
    In Section~\ref{LinfConv}, we will consider these definitions after a time rescaling with the constant factor $\frac{1}{c_A}$.
\end{remark*}

The proofs in Section~\ref{TL-convergence} are mostly based on minimization problems associated with converging energies. To obtain their convergence and to have the correct notation for the convergence of the energies, we need to define what we mean by $\Gamma$-convergence.

\begin{defi}[$\Gamma$-convergence]
    Let $X$ be a topological space and $(F_n)_{n\in\mathbb{N}}$ be a sequence of functionals $F_n:X\to\mathbb{R}\cup\{\infty\}$.
    Then, we say that $(F_n)_{n\in\mathbb{N}}$ $\Gamma$-converge to a functional $F:X\to\mathbb{R}\cup\{\infty\}$, written $F_n\overset{\Gamma}{\to}F$ if
    \begin{itemize}
        \item for all $x_n\to x, x_n,x\in X$ we have the liminf inequality
        $$\liminf\limits_n F_n(x_n)\geq F(x),$$
        \item for all $x\in X$ there exists a recovery sequence $x_n\to x, x_n\in X$ s.t.\ 
        $$\lim\limits_n F_n(x_n)= F(x).$$
    \end{itemize}
\end{defi}

\begin{remark*}
    The second property is often written as the equivalent formulation of a limsup inequality
    $$\limsup\limits_n F_n(x_n)\leq F(x).$$
\end{remark*}

Together with a compactness result of the family of energies, this will give us the fundamental theorem of $\Gamma$-convergence from which we can infer the convergence of the minimizers of the energies.

\begin{defi}[$\Gamma$-compactness]
    For a sequence $(F_n)_{n\in\mathbb{N}}$ of functionals, we say they are ($\Gamma$-)compact if for every sequence $(x_n)_{n\in\mathbb{N}}$, from the fact that $\limsup\limits_n F_n(x_n)<\infty$ it follows that $(x_n)_{n\in\mathbb{N}}$ is precompact.
\end{defi}

\section{\texorpdfstring{$L^\infty$}{Uniform} convergence \texorpdfstring{---}{-} Median Filter}
\label{LinfConv}

We will show the uniform convergence, in time and space, of our evolution to the viscosity solution, see Theorem~\ref{Cor:Conv}. 
The idea is to split this proof into two parts, a deterministic approximation and the randomness comparison.

Specifically, we will show that our scheme $u^n_h$ given by the median filter described in Section~\ref{Algorithm} converges to the level set solution $u(\cdot, t)$ of mean curvature flow (as $N\to \infty$ and $r,h\to 0$). This solution is given (under a time-rescaling) by the viscosity solution (Definition~\ref{def:viscositysolution}) to the evolution 
$$\partial_t u=c_AF(\nabla u, \nabla^2u)$$ 
for some $c_A>0$.
The RHS represents the surface Laplacian on the tangent plane. This scheme can be interpreted as a one-step direct Euler scheme with a discretization of the surface Laplacian:
$$\partial_t u\approx \frac{G_hu-u}{h}=\frac{m_Nu-u}{r^2}\approx c_AF(\nabla u, \nabla^2u).$$

The constant $c_A$ depends on the stencil (the local neighborhood) $A_r$ we use. In \cite{Oberman} and \cite{esedoglumedianfilter}, the sphere is used and there the constant is $\frac{1}{2}$ in dimension $d=2$. Note that for the $(d-1)$-dimensional sphere in $\mathbb{R}^d$ the constant would be $\frac{1}{2(d-1)}$.
We will use annuli (spherical shells) and as special examples, we will consider the ball $B_r$, a fixed annulus $A_{\kappa, r}\coloneqq B_r\setminus B_{\kappa r}$ and shrinking annuli, i.e., $A_{\kappa(r), r}$ with $\kappa(r)\to 1$ as $r\to 0$.
For the ball, the constant is $\frac{1}{2(d+1)}$ while for general annuli, it is $\frac{1}{2(d+1)}\frac{1-\kappa^{d+1}}{1-\kappa^{d-1}}$ and for the shrinking annuli $\frac{1}{2(d-1)}$ which coincides with the sphere as the limit.

The structure of our proof is as follows. First, in Section~\ref{ContinuousCase}, we show that the continuous mean field version of our scheme converges to the viscosity solution uniformly. In this version, the median of the points from the point process is replaced by a continuous median which is taken w.r.t.\ the uniform measure on $D$. 
Afterwards, in Section~\ref{DiscreteCase}, we prove that the discrete evolution relaying on the point process remains uniformly close to the continuous one and thus converges if the point process is fine enough. For details, see Theorem~\ref{Cor:Conv}.

\subsection{Continuous case}
\label{ContinuousCase}

For the deterministic approximation, we consider the mean field limits where the sampling is replaced by a uniform distribution over our domain $D\coloneqq\mathbb{T}^d$ (see Section~\ref{Introduction} for a discussion of other domains).
In this case, we prove the consistency for the distributional median $m_\mathbb{P}$ over each neighborhood instead of the empirical median. The definition of $m_\mathbb{P}$ for a random variable $X$ is given as an element of the set of values s.t.
$$\mathbb{P}(m_\mathbb{P}\geq X)\geq \frac{1}{2},\quad \mathbb{P}(m_\mathbb{P}\leq X)\geq \frac{1}{2}.$$
Here, the LHS in a local neighborhood $A$ becomes
$$\fint\limits_A \mathds{1}_{\{X\leq m_\mathbb{P}\}}\dd x.$$
The median also satisfies the continuous version of the variational characterization: $m_\mathbb{P}\in \argmin\limits_m \mathbb{E}[|X-m|]$. Usually, we will consider the infimum of the possible values of the median.
To prove the approximation, we will use the convergence result of \cite{BarlesSouganidis} with the techniques of \cite{IPS}. The latter theorems could also be applied in our case for the fixed kernels $B_r$ or $A_{\kappa, r}$ with fixed $\kappa$ but do not apply for varying kernels. Therefore, we will present the proof in our case which will also gives the exact time scales directly.

The approach of \cite{IPS} can be structured as follows. First, the monotonicity and a linear relabeling of the scheme must be verified. Next, the half-relaxed limits of our scheme must converge to the semi-continuous envelopes of the differential operator given by the viscosity solutions to level set mean curvature flow. By monotonicity, this consistency has to be verified only for smooth functions $\varphi\in C^\infty(D;\mathbb{R})$. Finally, the evolution must satisfy an a priori continuity estimate in time w.r.t.\ the initial datum. Combining these results, one obtains the uniform convergence.

To make these assumptions for the convergence theorem of \cite{BarlesSouganidis} (in our version \cite[Theorem~1.1]{IPS}) precise: We will require that our scheme is monotone and is translation invariant in the relabeling, i.e., for all $h>0, c\in\mathbb{R}$ and $ u,v\in \BUC(D)$ (see Proposition~\ref{Prop:Monotonicity}):
\begin{align}
   u\leq v &\Rightarrow G_h u\leq G_h v,\label{Monotonicity2}
   \\
    G_h(u+c)&=G_h u+c.\label{Monotonicity1}
\end{align}
In addition, the half-relaxed limits of the operator must converge to the upper- and lower semi-continuous envelopes of our continuous operator from the viscosity solution, i.e., for all $\varphi\in C^\infty(D)$:
\begin{align}
    \ulim\limits_{h\to 0}\frac{G_h\varphi-\varphi}{h}(x)&\leq F^*(\nabla\varphi(x), \nabla^2\varphi(x)),\label{Consistency1}\\
    \llim\limits_{h\to 0}\frac{G_h\varphi-\varphi}{h}(x)&\geq F_*(\nabla\varphi(x), \nabla^2\varphi(x)).\label{Consistency2}
\end{align}
We will show these two properties in Proposition~\ref{Prop:Consistency}. In this assumption, we have used the following two definitions.
\begin{defi}[Half-relaxed limits and semi-continuous envelopes]
    We define the upper and lower half-relaxed limit of a sequence $f_h$ via
    $$\ulim\limits_{h\to 0} f_h(x)\coloneqq \lim\limits_{\varepsilon\to 0}\sup\left\{f_h(y): h<\varepsilon, y\in B_\varepsilon(x)\right\}$$
    and
    $$\llim\limits_{h\to 0} f_h(x)\coloneqq \lim\limits_{\varepsilon\to 0}\inf\left\{f_h(y): h<\varepsilon, y\in B_\varepsilon(x)\right\}$$
    respectively.

    Moreover, the upper semi-continuous envelope of a function $f(x)$ is defined as 
    $$f^*(x)\coloneqq \lim\limits_{\varepsilon\to 0}\sup\{f(y):y\in B_\varepsilon(x)\}$$
    and the lower semi-continuous envelope as
    $$f_*(x)\coloneqq \lim\limits_{\varepsilon\to 0}\inf\{f(y):y\in B_\varepsilon(x)\},$$
    respectively.
\end{defi}
Recall, that $F$ is given by the level set Laplacian or, for singular points, its eigenvalues as in Definition~\ref{Def:F}.

Finally, the evolution must satisfy a continuity in time with respect to the initial datum. More precisely, for all $g\in \BUC(D), h>0$ there exists a modulus of continuity function $\omega\in C(\mathbb{R}_{\geq 0};\mathbb{R}_{\geq 0})$ depending only on $\omega_g$ and satisfying $\omega(0)=0$ s.t.\ for all $t\in [0,T]$, we have
\begin{align}
    \norm{Q_{h,t}^\mathbb{P}g-g}_{\infty}&\leq \omega(t).\label{ContinuityTime}
\end{align}
Here, $Q_{h,t}^\mathbb{P}$ is the piecewise constant-in-time evolution of our space-continuous non-local scheme defined by Equation~\eqref{Eq:DefQ}.
This property is proved as Proposition~\ref{Prop:TimeContinuity} using \cite[Lemma~3.2]{IPS} (our Lemma~\ref{Lemma:3.2}).

We will prove these assumptions and use them for the following theorem. This theorem demonstrates the convergence to the time continuous solution.

\begin{lemma}[Uniform convergence of the scheme {\cite[Theorem~1.1]{IPS}}]
    Let the initial data $g$ be in $\BUC(D)$ and $u\in \BUC(D\times [0,\infty))$ be the unique viscosity solution of (time-rescaled) level set mean curvature flow. Moreover, assume that \eqref{Monotonicity2}, \eqref{Monotonicity1}, \eqref{Consistency1}, \eqref{Consistency2}, and \eqref{ContinuityTime} hold. 

    Then, for fixed $T>0$, we have uniform convergence in $h\to 0$:
    $$Q^\mathbb{P}_{h,\cdot}g\to u \text{ uniformly in }D\times [0,T].$$
    \label{UniformConvergence}
\end{lemma}
\vspace*{-1.5em}
\begin{remark}
    The time rescaling by $\frac{1}{c_A}$ depends on the local neighborhood in the scheme. More precisely, for the kernel $A_{\kappa, 1}$, the time derivative $\partial_t$ becomes $2(d+1)\frac{1-\kappa^{d-1}}{1-\kappa^{d+1}}\partial_t$. In the case of the ball, this constant becomes $2(d+1)$ and for shrinking annuli (the sphere) $2(d-1)$.
    \label{Rem:constantCA}
\end{remark}

The main ingredient in the proof of the consistency at singular points and for \eqref{ContinuityTime}, will be the following lemma.

\begin{lemma}[{\cite[Lemma~3.2]{IPS}}]
\label{Lemma:3.2}
   There exist constants $C>0, \delta\in (0,h_0)$ s.t.\ for the test functions $\psi_
   \pm(x)\coloneqq \pm\sqrt{1+|x|^2}$ and all $x\in D, h\in (0,\delta]$, we have:
   \begin{align*}
       G_h\psi_+(x)&\leq \psi_+(x)+Ch,\\
       G_h\psi_-(x)&\geq \psi_-(x)-Ch.
   \end{align*}
\end{lemma}

Since $\psi$ behaves like $|x|$, this lemma enables us to prove the following.
\begin{lemma}[Continuity in time {\cite[Lemma~3.3]{IPS}}]
\label{Prop:TimeContinuity}
    Let $g\in \BUC(D)$ be a given initial datum. Then there exist $\delta\in (0,h_0)$ and a modulus of continuity function $\omega\in C([0,\infty);[0,\infty))$ satisfying $\omega(0)=0$ and depending only on the modulus of continuity $\omega_g$ of $g$ s.t.\ for all $ x\in D$, $t\geq 0$, and $h\in (0,\delta):$
    $$|Q_{h,t}^\mathbb{P}g(x)-g(x)|\leq \omega(t).$$
\end{lemma}

The following proof is due to \cite{IPS}.

\begin{proof}
    Let $g\in \BUC(D)$ given and fix $\varepsilon>0$. Then there exists $C_\varepsilon>0$ s.t.\ for all $x,y\in D$
    $$|g(x)-g(y)|\leq \omega_g(|x-y|)\leq \varepsilon+C_\varepsilon(\psi(x-y)-1)$$
    with $\psi(x)\coloneqq \sqrt{1+|x|^2}$.
    This implies
    $$-\varepsilon-C_\varepsilon(\psi(x-y)-1)+g(y)\leq g(x)\leq g(y)+\varepsilon+C_\varepsilon(\psi(x-y)-1).$$
    Now, we can apply $G_h$ (in $x$) to achieve with Lemma~\ref{Lemma:3.2}
    \begin{align*}
        -\varepsilon-C_\varepsilon(\psi(x-y)-1)-CC_\varepsilon h+g(y)\leq G_hg(x),\\
        G_hg(x)\leq g(y)+\varepsilon+C_\varepsilon(\psi(x-y)-1)+CC_\varepsilon h.
    \end{align*}
    An iteration and plugging in $y=x$ yields
    $$-\varepsilon-CC_\varepsilon t+g(x)\leq Q_{h,t}^\mathbb{P}g(x)\leq g(x)+\varepsilon+CC_\varepsilon t.$$
    The function $\omega(t)\coloneqq \inf\limits_{\varepsilon>0}\{\varepsilon+CC_\varepsilon t\}$ has the desired properties, continuity and vanishing at zero, and fulfills our estimate.
\end{proof}

\begin{prop}[Monotonicity and relabeling]
    \label{Prop:Monotonicity}
    The evolution $G_h u$ satisfies \eqref{Monotonicity2} and \eqref{Monotonicity1}.
\end{prop}

\begin{proof}
    This proposition follows immediately from the definition, since the median satisfies the same properties. If we consider $u+c$, then the median satisfies
    $$\med(u(x_1)+c,\dots, u(x_n)+c)=\med(u(x_1),\dots, u(x_n))+c.$$
    Similarly, the same holds for the monotonicity. Consider $u\leq v$, then 
    $$\med(u(x_1),\dots, u(x_n))\leq\med(v(x_1),\dots, v(x_n)).$$
\end{proof}

\begin{remark*}
    Both of these properties hold for the level set mean curvature flow. We even have a (monotone) relabeling property.
\end{remark*}

The stability of our scheme is a consequence of the maximum principle (monotonicity) and the relabeling property. Let $u\in \BUC(D)$ and $v$ be a smooth function approximating $u$ in $L^\infty(D)$. Then, there exists an $\varepsilon>0$ s.t.\ 
$$u+\varepsilon>v.$$ 
Hence, we also have for the evolution
$$G_h(u+\varepsilon)=u^1_h+\varepsilon>v^1_h.$$ 
By induction, this yields 
$$G^n_h(u+\varepsilon)=u^n_h+\varepsilon>v^n_h$$
and analogously
$$u^n_h-\varepsilon<v^n_h.$$ 
This implies that we can restrict ourselves to smooth functions for the consistency proof, since the approximation remains uniformly close in $L^\infty$ norm.\\
The same reasoning holds true for viscosity solutions to mean curvature flow.

Next, we will compute the median $m_\mathbb{P}$ pointwise for a given function and show that it is close to the level set Laplacian in a viscosity sense. Since we will use this for the consistency proof, we assume $\varphi\in C^\infty(D)$.

\begin{prop}[Consistency]
\label{Prop:Consistency}
    The evolution $G_h \varphi$ satisfies \eqref{Consistency1} and \eqref{Consistency2}.
\end{prop}

\begin{proof}
    The proof is split in two parts. First, we consider the non-degenerate case $\nabla \varphi\neq 0$. For this, we take a few simplifications and reformulations to compute the median explicitly up to a small error. Afterwards, we split the domain around the median level set and compute the median. In this computation a constant is involved which we calculate explicitly in the last step. The other case $\nabla \varphi=0$ is done by a reduction to an explicit function computation.
    
    \textbf{Next, we will assume that $\nabla \varphi\neq 0$}:\\
    Simplification: Fix $x\in D$ and let $A_r$ be the annulus $B_r\setminus B_{\kappa r}$ for $0\leq \kappa(r)<1$. With $\kappa=0$, we recover the ball, $\kappa\equiv \text{const.}$ represents the case of a fixed annulus and $\kappa\to 1$ the case of shrinking annuli. Now, we construct a function $\varphi$ with a few simplifications as described in Figure~\ref{Fig:Simplifications}. The median only changes in a linear way with these. To be precise, we shift and rotate the space by a relabeling as our scheme and MCF are translation and rotation invariant. Additionally, we normalize the function by a shift and scaling in the range and approximate it locally by its Taylor expansion.
    \begin{figure}[!ht]
    \label{Fig:Simplifications}
    \centering
    \begin{tabular}{ |c|p{60mm}| } 
    \hline
    Simplification & effect \\
    \hline\hline
    translate space & $x_0=0$, \newline $A_r=B_r(0)\setminus B_{\kappa r}(0),$\\
    relabel: $\varphi+c$ & $ \varphi(0)=0,$\\
    rotate space & $\nabla \varphi(0)\in \spn (e_d),$\\
    relabel: $cu$ & $\nabla \varphi(0)=e_d,$\\
    approximate by Taylor & $\varphi=a_{ij}x_ix_j+x_d+b_ix_ix_d+b_dx_d^2$\newline (up to $\mathcal{O}(r^3)$).\\
    \hline
    \end{tabular}
    \caption{The simplifications we employ.}
    \end{figure}

    In this proof, we use Einstein's notation to sum over repeated indices $i,j$ from $1$ up to $d-1$.
    
    Here, the rotation does not change $A_r$ due to its rotational symmetry. Note that the constant $c$ in the second relabeling depends on $|\nabla \varphi|$ and blows up as the gradient vanishes. Thus, we only achieve local convergence results.
    
    In this setting, we have $F(\nabla \varphi, \nabla^2\varphi)=|\nabla \varphi|\nabla\cdot \frac{\nabla \varphi}{|\nabla \varphi|}=2\sum\limits_{i=1}^{d-1}a_{ii}$.
    
    The median $m_\mathbb{P}$ is given by the equation (by Sard's theorem the level sets are negligible if $|\nabla \varphi|>0$):
    \begin{align*}
        \fint\limits_{A_r(0)}\mathds{1}_{\{\varphi(x)<m_\mathbb{P}\}}\dd x=\frac{1}{2}.
    \end{align*}
    By a rescaling/normalization ($x\mapsto rx$) and another relabeling for 
    \begin{align*}
        \widetilde m_\mathbb{P}&\coloneqq \frac{m_\mathbb{P}}{r^2},\\
        f(x)&\coloneqq a_{ij}x_ix_j+\frac{x_d}{r}+b_ix_ix_d+b_dx_d^2,
    \end{align*}
    we get:
    \begin{align*}
        \fint\limits_{A_1(0)}\mathds{1}_{\{f(x)<\widetilde m_\mathbb{P}\}}\dd x=\frac{1}{2}.
    \end{align*}
    Splitting around the median level set: Let $C\coloneqq \sum\limits_{i,j=1}^{d-1}|a_{ij}|+B>0, B\coloneqq \sum\limits_{i=1}^d|b_i|\geq 0$. With these constants, one can see $\left|f-\frac{x_d}{r}\right|\leq C$ and thus 
    $$f\in \left[\frac{x_d}{r}-C, \frac{x_d}{r}+C\right].$$
    Next, we want to use the symmetry of $A_1$ and the fact that $r$ is small, to restrict the rescaled $x_d$ to an interval of size $\mathcal{O}(r)$ for the computation of the median. Let $r$ be small enough (such that $6Cr<1$), define $b\coloneqq 3Cr$ and observe
    \begin{align*}
        \{x_d>4Cr\}&\subseteq \{f>3C\} \subseteq \{x_d>2Cr\}\\
        \Rightarrow \mathbb{P}(\{x_d>4Cr\})&\leq \mathbb{P}(\{f>3C\})\leq \mathbb{P}(\{x_d>2Cr\})\\
        \Rightarrow \mathbb{P}(\{x_d<-4Cr\})&\leq \mathbb{P}(\{f>3C\})\leq \mathbb{P}(\{x_d<-2Cr\})\\
        \Rightarrow \mathbb{P}(\{f<-5C\}))&\leq \mathbb{P}(\{f>3C\})\leq \mathbb{P}(\{f<-C\}).        
    \end{align*}
    Hence, by monotonicity and continuity ($|\nabla f|>0$) there exists an $a\in(-5Cr, -Cr)$ s.t.\ 
    $$\mathbb{P}(\{f<a\})=\mathbb{P}(\{f>b\}).$$
    Since $a<b$, $\widetilde m_\mathbb{P} f=\widetilde m_{\mathbb{P}\llcorner_{\{a<f<b\}}}f$. Therefore, we can conclude $|\widetilde m_\mathbb{P}|<5C$ as $|a|,|b|<5Cr$. This allows us to restrict the neighborhood $A_1$ to $A_1\cap \{|x_d|<6Cr\}$ since $\mathbb{P}(\{x_d<-6Cr\})=\mathbb{P}(\{x_d>6Cr\})$.
    
    A further relabeling of $x_d$ and the restriction yield the new condition with $\widetilde f(x)\coloneqq a_{ij}x_ix_j+x_d+rb_ix_ix_d+r^2b_dx_d^2$:
    \begin{align*}
        \fint\limits_{A_1(0)\cap \{|x_d|<6C\}}\mathds{1}_{\{\widetilde f(x)<\widetilde m_\mathbb{P}\}}\dd x=\frac{1}{2}.
    \end{align*}
    Now, we want to change the domain to an annulus-cylinder as this will simplify the calculations. In addition, we drop the constant $6C$ since it only rescales all calculations.

    For readability, we first show the proof for a ball neighborhood and then adapt the proof to the case of general annuli.
    
    The ball restricted to $\{|x_d|<r\}$ is contained in the cylinder (in direction $x_d$) with radius $1$ and contains the cylinder with radius $\sqrt{1-r^2}$. In the rescaled setting, the volume difference between both is 
    \begin{align*}
        2\omega_{d-1}\left(1-\sqrt{1-r^2}^{d-1}\right)=2\omega_{d-1}\frac{1-(1-r^2)^{d-1}}{1+\sqrt{1-r^2}^{d-1}}\in \mathcal{O}(r^{2(d-1)}).
    \end{align*}
    Thus, it suffices to show that $\widetilde f$ remains uniformly within an $\mathcal{O}(r)$ margin in a sufficiently large area. Then, the median only changes by this margin if the volume and values are changed. Let $B_1$ be the (cut) ball and $C_1$ the cylinder. Suppose that for $m$, the median of $f$ in $B_1$, we have
    \begin{align*}
    |B_1\Delta C_1|&\leq \delta_2,\\
    |\{f>m-\varepsilon\}\cap B_1|&\geq \frac{|B_1|}{2}+\delta_1,\\
    |\{f<m+\varepsilon\}\cap B_1|&\geq \frac{|B_1|}{2}+\delta_1.
    \end{align*}
    where $B_1\Delta C_1\coloneqq B_1\setminus C_1\cup C_1\setminus B_1$ denotes the symmetric difference.
    Thus,
    \begin{align*}
        |\{f>m-\varepsilon\}\cap C_1|&\geq |\{f>m-\varepsilon\}\cap B_1|-|B_1\setminus C_1|\\
        &\geq \frac{|B_1|}{2}+\delta_1-\delta_2\\
        &\geq \frac{|C_1|}{2}-\frac{|C_1\setminus B_1|}{2}-\delta_2+\delta_1\\
        &\geq \frac{|C_1|}{2}+\delta_1-\frac{3}{2}\delta_2.
    \end{align*}
    Analogously for $\{f<m-\varepsilon\}$. This means that if $\delta_1\geq \frac{3}{2}\delta_2$, then the median of $f$ in $C_1$ remains in a $\varepsilon$-margin compared to the median of $f$ in $B_1$.
    
    By using the quadratic structure of $\widetilde f$, we can apply Lemma~\ref{tauEstimate} combined with Proposition \ref{MinkowskiContent} to achieve that there are sets of size $\Omega(r)$ s.t.\ $0>\widetilde f(y)-\widetilde m_\mathbb{P}>-c r$ and analogously for the other direction.
    
    Hence, $\widetilde f$ stays in an $\mathcal{O}(r)$ margin in an area of size $\Omega(r)$ around the set of points which represent the median. Therefore, we can choose $\varepsilon\coloneqq c_1r, \delta_1\coloneqq c_2r, \delta_2\coloneqq c_3r^{2(d-1)}$. It would even suffice to use the weaker and easier bound $\delta_1=c_2r^d$.
    Thus, in our desired accuracy, we can go to the cylinder in direction $x_d$.
    
    Median computation: We are now able to compute the integral explicitly using a simple estimate
    \begin{align*}
        a_{ij}x_ix_j+x_d-rB&\leq \widetilde f\leq a_{ij}x_ix_j+x_d+rB,\\
        \mathds{1}_{\{a_{ij}x_ix_j+x_d-rB<\widetilde m_\mathbb{P}\}}&\geq \mathds{1}_{\{\widetilde f<\widetilde m_\mathbb{P}\}}\geq \mathds{1}_{\{a_{ij}x_ix_j+x_d+rB<\widetilde m_\mathbb{P}\}}.
    \end{align*}
    This allows to compute (writing $x'\coloneqq (x_1,\dots, x_{d-1})$)
    \begin{align*}
        \fint\limits_{A_1^{d-1}}\fint\limits_{-1}^1&\mathds{1}_{\{a_{ij}x_ix_j+x_d\pm rB<\widetilde m_\mathbb{P}\}}\dd x_d\dd x'-\frac{1}{2}\\
        &=\frac{1}{2}\left(-\fint\limits_{A_1^{d-1}}a_{ij}x_ix_j\dd x'+\widetilde m_\mathbb{P}\mp rB\right).
    \end{align*}
    Thus,
    \begin{align*}
         \widetilde m_\mathbb{P}&=\fint\limits_{B_1^{d-1}}a_{ij}x_ix_j\dd x'+\mathcal{O}(r)\\
        &=\frac{\omega_{d-2}}{\omega_{d-1}}\sum\limits_{i=1}^{d-1}a_{ii}\int\limits_{-1}^1s^2\sqrt{1-s^2}^{d-2}\dd s+\mathcal{O}(r)\\
        &=\frac{1}{c}\sum\limits_{i=1}^{d-1}2a_{ii}+\mathcal{O}(r),\\
        \frac{2}{c}&=\frac{\omega_{d-2}}{\omega_{d-1}}\int\limits_{-1}^1s^2\sqrt{1-s^2}^{d-2}\dd s\\
        &=\frac{\omega_{d-2}}{\omega_{d-1}}\int\limits_{-1}^1\left(1-\sqrt{1-s^2}^2\right)\sqrt{1-s^2}^{d-2}\dd s\\
        &=1-\frac{\omega_{d-2}\omega_{d+1}}{\omega_{d-1}\omega_{d}}=1-\frac{d}{d+1}=\frac{1}{d+1}.
    \end{align*}
    In the last step, we used that $\frac{\omega_d}{\omega_{d-2}}=\frac{2\pi}{d}$.
    We conclude that the median is $\widetilde m_\mathbb{P}=\frac{1}{2(d+1)}\sum\limits_{i=1}^{d-1}2a_{ii}+\mathcal{O}(r)$. 
    
    This calculation is justified as by the normalization of the constant $6C$, we have $a_{ij}x_ix_j+\widetilde m_\mathbb{P}+rB\leq 1$.

    The proof can be adapted to fit our scheme with annuli instead of balls:\\
    The proof goes through the same with $A_1\coloneqq B_1\setminus B_\kappa$ instead of $B_1$ until the computation of the constant.
    \begin{align*}
        \widetilde m_\mathbb{P}&= \mathcal{O}(r)+\fint_{A_1^{d-1}} a_{ij}x_ix_j\dd x\\
        \overset{Symm.}&{=}\mathcal{O}(r)+\fint_{A_1^{d-1}} a_{ii}x_i^2\dd x.
    \end{align*}
    The latter term can be computed separately:
    \begin{align*}
        \sum\limits_{i=1}^{d-1} a_{ii}\frac{1}{(1-\kappa^{d-1})\omega_{d-1}}\int\limits_{-1}^1x_d^2\int\limits_{\mathbb{R}^{d-1}}\mathds{1}_{\{(x')^2+x_d^2\in [\kappa^2,1]\}}\dd x'\dd x_d.
    \end{align*}
    The volume is $\begin{cases}
        \omega_{d-2}\left((1-x_d^2)^\frac{d-2}{2}-(\kappa^2-x_d^2)^\frac{d-2}{2}\right) &\text{for }x_d<\kappa,\\
        \omega_{d-2}(1-x_d^2)^\frac{d-2}{2} &\text{for }x_d\geq \kappa.
    \end{cases}$
    
    Thus, we can compute
    \begin{align*}
        \int\limits_{-1}^1\int\limits_{\mathbb{R}^{d-1}}&\mathds{1}_{\{|x'|^2+x_d^2\in [\kappa^2,1]\}}\dd x'\dd x_d\\
        &=\frac{2\omega_{d-2}}{(1-\kappa^{d-1})\omega_{d-1}}\left(\int\limits_0^1 (1-s^2)^\frac{d-2}{2}\dd s-\int\limits_0^\kappa (\kappa^2-s^2)^\frac{d-2}{2}\dd s\right).
    \end{align*}
    Observe
    \begin{align*}
        \int\limits_0^\kappa (\kappa^2-x^2)^\frac{d-2}{2}\dd x&=\kappa^{d-1}\int\limits_0^1 (1-x^2)^\frac{d-2}{2}\dd x,\\
        \int\limits_0^1 (1-x^2)^\frac{d-2}{2}\dd x&=\frac{\omega_{d-1}}{2\omega_{d-2}}.
    \end{align*}
    This yields
    \begin{align*}
        \widetilde m_\mathbb{P}&= \mathcal{O}(r)+\sum\limits_{i\neq d} a_{ii}\cdot \frac{2\omega_{d-2}}{(1-\kappa^{d-1})\omega_{d-1}}\\
        &\quad\cdot\left(\int\limits_0^\kappa x^2((1-x^2)^\frac{d-2}{2}-(\kappa^2-x^2)^\frac{d-2}{2})\dd x+\int\limits_\kappa^1x^2(1-x^2)^\frac{d-2}{2}\dd x\right)\\
        &= \mathcal{O}(r)+\sum\limits_{i\neq d} a_{ii}\cdot \frac{2\omega_{d-2}}{(1-\kappa^{d-1})\omega_{d-1}}\\
        &\quad\cdot\left(\int\limits_0^1 x^2(1-x^2)^\frac{d-2}{2}\dd x-\kappa^{d+1}\int\limits_0^1x^2(1-x^2)^\frac{d-2}{2}\dd x\right).
    \end{align*}
    Using $x^2=1-(1-x^2)$ and the above formula we continue
    \begin{align*}
        \widetilde m_\mathbb{P}&= \mathcal{O}(r)+\sum\limits_{i\neq d} a_{ii}\cdot \frac{2\omega_{d-2}}{(1-\kappa^{d-1})\omega_{d-1}}(1-\kappa^{d+1})\\
        &\hspace*{7.3em}\cdot\left(\int\limits_0^1 (1-x^2)^\frac{d-2}{2}\dd x-\int\limits_0^1 (1-x^2)^\frac{d}{2}\dd x\right)\\
        &=\mathcal{O}(r)+\sum\limits_{i\neq d} a_{ii}\cdot \frac{2\omega_{d-2}}{(1-\kappa^{d-1})\omega_{d-1}}(1-\kappa^{d+1})\left(\frac{\omega_{d-1}}{2\omega_{d-2}}-\frac{\omega_{d+1}}{2\omega_d}\right).
    \end{align*}
    Lastly, we plug in $\frac{\omega_d}{\omega_{d-2}}=\frac{2\pi}{d}$:
    \begin{align*}
        \widetilde m_\mathbb{P}&=\mathcal{O}(r)+\frac{1}{d+1}\sum\limits_{i\neq d} 2a_{ii}\cdot\frac{1-\kappa^{d+1}}{1-\kappa^{d-1}}.
    \end{align*}
    Note that this is consistent with our result for the ball.
    In the limit $\kappa\nearrow 1$ of shrinking spheres this becomes
    \begin{align*}
        \widetilde m_\mathbb{P}&= \mathcal{O}(r)+\frac{1}{d-1}\sum\limits_{i\neq d} 2a_{ii}.
    \end{align*}

    Here, the pointwise convergence suffices for \eqref{Consistency1} and \eqref{Consistency2} as for $|\nabla \varphi|>0$ the same holds by continuity in a local neighborhood. Additionally, $F(\nabla \varphi, \nabla^2 \varphi)$ is continuous and our local uniform approximation result only depends on $|\nabla \varphi|^{-1}$. This guarantees the convergence.

    \textbf{If $\nabla \varphi=0$}, we furthermore can reduce the setting to $\varphi(x)=c |x|^4$ with $c\geq 0$ by \cite{IPS}. We will follow their proof. In this case, we can assume $\nabla \varphi=0, \nabla^2 \varphi=0, x_0=0$ and $F_*(0,0)=F^*(0,0)=0$. With $\psi(x)\coloneqq\sqrt{1+|x|^2}$, we see that $\varphi(x)=c (\psi(x)^2-1)^2$. Thus, we can apply Lemma~\ref{Lemma:3.2} with a suitable $1>\varepsilon>0, C>0$ (the exact value will change from line to line without relabeling) we have for all $x\in B_\varepsilon(0), h\leq \varepsilon$:
    \begin{align*}
        G_h\varphi (x)&=c ((G_h\psi)^2-1)^2\\
        &\leq c ((\psi+Ch)^2-1)^2\\
        &=c (|x|^2+(2\sqrt{1+|x|^2}C+C^2h)h)^2\\
        &\leq c (|x|^2+Ch)^2\\
        &=c|x|^4+c(C^2h+2|x|^2C)h\\
        &\leq \varphi(x)+c C\varepsilon h.
    \end{align*}
    Therefore,
    \begin{align*}
        \ulim\limits_{h\to 0}\frac{G_h\varphi-\varphi}{h}(0)\coloneqq\lim\limits_{\varepsilon\to 0}\sup\left\{\frac{G_h\varphi (x)-\varphi(x)}{h}: h<\varepsilon, y\in B_\varepsilon(0)\right\}&\leq 0.
    \end{align*}
    The other inequality follows analogously.
\end{proof}

\begin{remark*}
    In the case of the ball as local neighborhood $A_r=B_r$, the pointwise convergence does not hold at points where $\nabla u=0$, even in the sense of the relaxed viscosity solution.
\end{remark*}

In the second part of this proof, we used Lemma~\ref{Lemma:3.2}. With the technique of the first part of the proof, we are now able to prove this.

\begin{proof}[Proof of Lemma~\ref{Lemma:3.2}]
    We consider $\psi(x)\coloneqq \sqrt{1+|x|^2}$. Now, the median can be reformulated as
    \begin{align*}
        \frac{1}{2}&=\int\limits_{A_r(x_0)}\mathds{1}_{\{\psi(x)<m_\mathbb{P}\}}\dd x.
    \end{align*}
    By symmetry, this is equivalent to (here, w.l.o.g.\ and by abuse of notation $x_0=x_0e_d$ with $x_0\geq 0$)
    \begin{align*}
        \frac{1}{2}&=\int\limits_{A_r(0)}\mathds{1}_{\{|x|^2+\frac{2x_0}{r}x_d<\widetilde m_\mathbb{P}(\sqrt{1+|x_0 e_d+rx_d|^2}+\sqrt{1+x_0^2})\}}\dd x.
    \end{align*}

    With $f(x)\coloneqq |x|^2+\frac{2x_0}{r}x_d$, we have
    $$\{x_d>c\frac{r}{2x_0}\}\subseteq \{f>c\}\subseteq \{x_d>(c-1)\frac{r}{2x_0}\}$$
    and
    $$\{x_d<(c-1)\frac{r}{2x_0}\}\subseteq \{f<c\}\subseteq \{x_d<c\frac{r}{2x_0}\}.$$

    In the same fashion as before, 
    \begin{align*}
        \mathbb{P}(\{f<0\})&\geq \mathbb{P}(\{x_d<-\frac{r}{2x_0}\})\\
        &=\mathbb{P}(\{x_d>\frac{r}{2x_0}\})\\
        &\geq \mathbb{P}(\{f>2\})\\
        &\geq \mathbb{P}(\{x_d>\frac{r}{x_0}\})\\
        &=\mathbb{P}(\{x_d<-\frac{r}{x_0}\})\\
        &\geq \mathbb{P}(\{f<-2\}).
    \end{align*}
    
    Using monotonicity and continuity, we conclude that there exists $-2\leq a\leq 0$ s.t.\ 
    $$\mathbb{P}(\{f>2\})=\mathbb{P}(\{f<a\}).$$
    Thus, the calculation of the median can be restricted to $a\leq f\leq 2$ and it is bounded from above by $2$. Note that it can happen that these probabilities are zero, e.g., when $x_0=0$ or in general when $x_0\ll 1$. In this case the restriction is not necessary since the bound is trivially satisfied.
    Now, we know that 
    $$\widetilde m_\mathbb{P}\leq \frac{2}{\sqrt{1+|x_0 e_d+rx_d|^2}+\sqrt{1+x_0^2}}\leq 1.$$

    For $-\psi(x)= -\sqrt{1+|x|^2}$, we know that if 
    $$\widetilde m_\mathbb{P}=\frac{G_h\psi-\psi}{h}\leq C$$
    then 
    $$\frac{G_h(-\psi)-(-\psi)}{h}=-\widetilde m_\mathbb{P}\geq -C.\qedhere$$
\end{proof}

We are now in a position to use Lemma~\ref{UniformConvergence}. By Proposition~\ref{Prop:Consistency}, Proposition~\ref{Prop:Monotonicity} and Proposition~\ref{Prop:TimeContinuity}, our scheme satisfies the assumptions of Lemma~\ref{UniformConvergence} which implies the uniform convergence to the viscosity solution of level set mean curvature flow.

\begin{remark}[Counterexample to \cite{Oberman}]
    In~\cite{Oberman}, Oberman claims for the sphere
    $$\frac{m_\mathbb{P}-u(x_0)}{r^2}-F(\nabla u, \nabla^2u)=\mathcal{O}(r^2).$$
    Such a global uniform statement does not hold but can be replaced by a pointwise weaker convergence as can be seen in this paper and in \cite{esedoglumedianfilter}. The mistake stems from a Taylor approximation that increases the order of approximation afterwards.
    
    The counterexample is $u(x)=x_1^2$ in $\mathbb{R}^2$.\\
    Here, $\nabla u=2x_1e_1$ and $\nabla^2 u=\begin{pmatrix}
        2&0\\0&0
    \end{pmatrix}$. 
    This implies for $F$ as in Definition~\ref{Def:F}
    $$F(\nabla u, \nabla^2u)=\begin{cases}
        0 & \text{ for }x_1\neq 0,\\ [0,2] & \text{ for }x_1=0.
    \end{cases}$$
    The median can be computed by solving:
    \begin{align*}
        \frac{1}{2}&=\fint\limits_{\partial B_r(x_0e_1)}\mathds{1}_{\{x_1^2<m_\mathbb{P}\}}\dd x\\
        &=\frac{1}{2\pi}\int_0^{2\pi}\mathds{1}_{\{\cos^2(\theta)+\frac{2x_0\cos(\theta)}{r}\leq\widetilde m_{\mathbb{P}}\}}\dd \theta.
    \end{align*}
    For $a=\frac{r}{10}$ this is yields that $\widetilde m_{\mathbb{P}}$ is a strictly positive constant, $\widetilde m_{\mathbb{P}}\approx 0{.}48$ which contradicts the claim that would be:
    $$\widetilde m_{\mathbb{P}}=\mathcal{O}(r^2).$$
\end{remark}

\subsection{Discrete Case}
\label{DiscreteCase}

In this section, we compare the continuous median $m_\mathbb{P}$ to the discrete $m_N$ and show that with high probability which goes to $1$ as $r\to 0$ both evolutions are uniformly close and therefore, they are also close to the viscosity solution as shown in Section~\ref{ContinuousCase}.

For this, we use the Dvoretzky-Kiefer-Wolfowitz inequality, see~\cite{DKW}. It states a convergence of the cumulative distribution functions (CDF) by approximations of finitely many samples.

\begin{lemma}[Dvoretzky-Kiefer-Wolfowitz inequality]
    \label{Lemma:DKW}
    Let $X_i$ be i.i.d.\ sampled random variables with CDF $F$. For $F_X(x)\coloneqq \mathbb{P}(X\leq x)$ and $F_N(x)\coloneqq \frac{1}{N}\sum\limits_{i=1}^N\mathds{1}_{\{X_i\leq x\}}$, we have convergence in probability:
    \begin{align*}
        \forall \varepsilon>0:  \mathbb{P}_N(\sup\limits_x|F_N(x)-F_X(x)|>\varepsilon)\leq 2e^{-2N\varepsilon^2}.
    \end{align*}
\end{lemma}

In this setting, we can define the medians.
\begin{align*}
    m_\mathbb{P}&\coloneqq \inf\left\{m:F_X(m)\geq \frac{1}{2}\right\},\quad F_X(m_\mathbb{P})=\frac{1}{2},\\
    m_N&\coloneqq \inf\left\{m:F_N(m)\geq \frac{1}{2}\right\},\quad F_N(m_N)=\begin{cases}\frac{1}{2}&N \text{ even},\\\frac{1}{2}+\frac{1}{2N}&N \text{ odd}.\end{cases}
\end{align*}
With $\tau_1\coloneqq \frac{1}{2}-F_X(m_\mathbb{P}-\varepsilon), \tau_2\coloneqq F_X(m_\mathbb{P}+\varepsilon)-\frac{1}{2}$ this yields:
\begin{Corollary}
   For the distance of the discrete and continuous median, we have the following estimate:
    \begin{align*}
        \mathbb{P}_N(|m_N-m_\mathbb{P}|>\varepsilon)\leq e^{-2N\tau_1^2}+2e^{2\tau_2}e^{-2N\tau_2^2}.
    \end{align*}
\end{Corollary}
\begin{proof}
    Fix $\varepsilon>0$. We want to estimate $\mathbb{P}_N(|m_N-m_\mathbb{P}|>\varepsilon)$.

    \begin{align*}
        \mathbb{P}_N(|m_N-m_\mathbb{P}|>\varepsilon)&\leq \mathbb{P}_N(|m_N-m_\mathbb{P}|\geq \varepsilon)\\
        &=\mathbb{P}_N(m_N\leq m_\mathbb{P}-\varepsilon)+\mathbb{P}(m_N\geq m_\mathbb{P}+\varepsilon)\\
        &=:\text{I}+\text{II}.
    \end{align*}
    Now, both parts can be estimated using the DKW Theorem (Lemma~\ref{Lemma:DKW}) and monotonicity ($\chi$ is $1$ iff $N$ is odd):
    \begin{align*}
        \text{I}&\leq \mathbb{P}_N(F_N(m_N)\leq F_N(m_\mathbb{P}-\varepsilon))\\
        &=\mathbb{P}_N\left(\frac{1}{2}+\frac{1}{2N}\chi-F_X(m_\mathbb{P}-\varepsilon)\leq F_N(m_\mathbb{P}-\varepsilon)-F_X(m_\mathbb{P}-\varepsilon)\right)\\
        &\leq \mathbb{P}_N(\tau_1\leq \sup F_N-F_X)\\
        &\leq e^{-2N\tau_1^2}.
    \end{align*}
    Analogously, 
    \begin{align*}
        \text{II}&\leq \mathbb{P}_N(F_N(m_\mathbb{P}+\varepsilon)\leq F_N(m_N))\\
        &=\mathbb{P}_N(-F_N(m_\mathbb{P}+\varepsilon)\geq -F_N(m_N))\\
        &=\mathbb{P}_N(F_X(m_\mathbb{P}+\varepsilon)-F_N(m_\mathbb{P}+\varepsilon)\geq F_X(m_\mathbb{P}+\varepsilon)-F_N(m_N))\\
        &\leq \mathbb{P}_N\left(\sup |F_X-F_N|\geq \tau_2-\frac{1}{2N}\chi\right)\\
        &\leq 2e^{-2N(\tau_2-\frac{1}{2N}\chi)^2}\leq 2e^{2\tau_2}e^{-2N\tau_2^2}.\qedhere
    \end{align*}
\end{proof}

Using this corollary, we see that it suffices to bound $\tau_{1/2}$ from below and sum up the probabilities over time and space.
This estimate is done using the modulus of continuity via isoperimetric arguments.

\begin{lemma}
\label{tauEstimate}
Let $g\in \BUC(D)$ with $\omega_g$ the modulus of continuity of $g$. Then,
$\tau_{1/2}\gtrsim \frac{\omega_g^{-1}(\varepsilon)}{r}$.
\end{lemma}

For the proof of this lemma, we will view $\tau$ as the mass of a difference of sets and bound it by a relative outer Minkowski content. Hence, we will need to bound this content.

\begin{prop}[Minkowski content]
\label{MinkowskiContent}
    For the normalized local neighborhood $A_1$ (a ball, a fixed annulus or a series of shrinking annuli) and $E\subseteq A_1$ with $|E|=\frac{1}{2}|A_1|$ as well as for $\delta_0>\delta>0$, we have
    $$\frac{|E_\delta\setminus E|}{|A_1|}\gtrsim \delta.$$
\end{prop}

\begin{proof}[Proof Proposition~\ref{MinkowskiContent} --- Part 1]
    In this part, we will prove the proposition for the ball and fixed annulus. The case shrinking annuli will be done in the second part.
    The procedure is as follows: First, we will use the relative isoperimetric inequality to bound the relative perimeter in terms of the mass of the set $E$. This allows us to obtain a lower bound on the outer Minkowski content.

    \cite[Theorem~3.1]{chambolle2014remark} ensures that if $E$ is not of finite perimeter in $A$ then the outer Minkowski content is infinite.

    For a fixed annulus $A_{\kappa, 1}$ or the ball one can use the relative isoperimetric inequality proved in \cite{Mazya}, since the annulus is the union of two parts that have quasi-isometric transformations to the ball. Here, we are only interested in $\delta$ which is small (but fixed) compared to the geometry of $A_1$ and we thus can restrict ourselves to the setting of $|E|, |E_\delta|\in \left(\frac{|A_1|}{10}, \frac{9|A_1|}{10}\right)$. For shrinking annuli, one has to prove the relative isoperimetric inequality per hand, see below. With an approximation argument, one obtains this inequality for any set of finite perimeter and thus for measurable sets. 
    Using the coarea formula with the first part of \cite[Proof of Theorem 5]{ambrosio2008outer}, we obtain a lower bound on the outer Minkowski content by the perimeter
    $$\limsup\limits_{\delta\to 0}\frac{|E_\delta\setminus E|}{\delta}\geq \Per(E;A_1).$$ Together with the isoperimetric inequality
    $$\Per(E;A_1)\geq c(A_1)|E|^\frac{d-1}{d},$$
    this can be used to non-localize the outer Minkowski content using the technique of Ledoux:
    \begin{lemma}[{\cite[Proposition~2.1]{ledoux2001concentration}}]
        Let $\mu$ be a.c.\ w.r.t.\ the Lebesgue measure in $\mathbb{R}^d$ and $I_\mu$ the isoperimetric function (the maximum lower bound of the outer Minkowski content by the measure of a set). Assume $I_\mu\geq v'\circ v^{-1}$ for a strictly increasing differentiable $v:I\to [0,\mu(X)]$ with an interval $I$. Then for every $r>0$
        $$v^{-1}(\mu(E_r))\geq v^{-1}(\mu(E))+r.$$
    \end{lemma}
    Using this, we get our desired result $v^{-1}(|E_\delta|)\geq v^{-1}(|E|)+\delta$ for $v(r)=c\cdot r^{d}$. By the Bernoulli inequality
    $$\frac{|E_\delta\setminus E|}{|A_1|}\gtrsim\delta.$$
    For the shrinking annuli, this proof does not work because the constants and volumes in the above proof either explode or vanish.
\end{proof}

\begin{remark}
    Without the normalization, the condition becomes $\delta_0\cdot r>\delta$. As we will see in the proof of Lemma~\ref{tauEstimate}, the best we can hope for and will use is $\delta\in \Omega(r)$ for an initial datum which is Lipschitz. For general $g$, we will have a smaller $\delta$. Thus, this condition can always be satisfied.
\end{remark}

We used the relative isoperimetric inequality for a fixed neighborhood $A_1$. Next, we show that a similar inequality holds for shrinking annuli.

\begin{prop}[Relative isoperimetric inequality in shrinking annuli]
\label{Prop:relIsopAnnuli}
    Let $A_{\kappa, 1}\coloneqq B_1\setminus B_\kappa$ be an annulus and $E\subseteq A_{\kappa, 1}$ a measurable, relatively open set with $\frac{|E|}{|A_{\kappa,1}|}\in \left(\frac{1}{10},\frac{9}{10}\right)$. Then, there exists a constant $c>0$ s.t.\ for all $1>\kappa>\kappa_0$, we have
    $$\Per(E;A_{\kappa,1})\geq c |E|.$$
\end{prop}

\begin{proof}
    In the following, we assume $|E|\leq \frac{1}{2}|A_{\kappa,1}|$ by taking $A_{\kappa,1}\setminus E$ instead of $E$ if necessary. For the result, this only changes the constant by a factor, since the restrictions on the relative mass of $E$ allow one to take a larger constant to get from $A_{\kappa,1}\setminus E$ to $E$, since their volumes are comparable.

    Using \cite[Chapter 7]{baernstein2019symmetrization}, we can use the (spherical) cap symmetrization w.r.t.\ $e_1$ which reduces the perimeter while maintaining the volume. Thus, for the isoperimetric inequality, we can simplify our problem to sets $E$ that are radially symmetric. Moreover, we can write $E$ depending on a one-dimensional function $f:[\kappa, 1]\to [0,2\pi]$ s.t.\ $E\cap \partial B_r$ is a spherical cap with center $re_1$ and (spherical) radius $f(r)$. Here, we take the convention that $f(r)=0$ means that this intersection is empty, since we are interested in minimizers. Next, using the parameterization of the unit sphere of \cite{aberra2007surfaces}, we parameterize this boundary via
    \begin{align*}
        p:[0,\pi]^{d-3}\times [0,2\pi]\times [\kappa,1]&\to \mathbb{R}^d,\\
        (r,\theta_1,\dots,\theta_{d-2})&\mapsto \vect{r \cos(f)\\ r\sin(f) v},\\
        \text{where }v&\coloneqq\vect{\cos \theta_1\\ \sin\theta_1\cos\theta_2\\\vdots\\\sin\theta_1\dots \cos \theta_{d-2}\\\sin\theta_1\dots \sin \theta_{d-2}}\in \mathbb{S}^{d-2}.
    \end{align*}
    In $d=2$, $v$ takes the values $\pm 1$.
    
    Now, the area factor $\sqrt{\det Dp^tDp}$ in the area formula is dependent on 
    $$Dp=\begin{pmatrix}
        \cos f-r\sin ff' &0\\ (\sin f+r\cos f f')v & r\sin f D_\theta v
    \end{pmatrix}.$$
    As $|v|=1$, the matrix $Dp^tDp$ is a diagonal matrix and since $v$ is the usual parameterization of the unit sphere, we conclude for $f\in C^1$:
    \begin{align*}
        \Per(E;A_{\kappa,1})&=(d-1)\omega_{d-1}\int\limits_\kappa^1 r^{d-2}\sin^{d-2}f(r)\sqrt{1+r^2f'^2(r)}\dd r.
    \end{align*}
    This is in accordance with the formula for spheres in \cite{bogelein2017quantitative}. For general $E$, we can have jumps that contribute as 
    $$(d-1)\omega_{d-1}\int\limits_\kappa^1 r^{d-1}\int\limits_{f^-(r)}^{f^+(r)}\sin^{d-2}(\sigma)\dd \sigma \dd\mathcal{H}^0(r).$$
    Here, $f^{\pm}$ denote the left and right limit of $f$ at $r$.
    This can also be seen as a limit of the smooth version. By approximating the set of finite perimeter $E$ with smooth sets, we can achieve an $f$ that is piecewise $C^1$ and only need to bound the smooth integral and the jump part.\\
    In the same fashion the mass constraint is computed
    $$|E|=(d-1)\omega_{d-1}\int\limits_\kappa^1 r^{d-1}\int\limits_0^{f(r)}\sin^{d-2}(\sigma)\dd\sigma\dd r.$$

    Fix $\theta=\pi\cdot \frac{9}{10}$. If there exists $r$ s.t.\ $f(r)\geq \theta$, then there exists an $r'$ s.t.\ $f(r')\leq \frac{\pi}{2}$ (w.l.o.g.\ $r'<r$ else switch them). This $r'$ exists as we assumed $|E|\leq \frac{1}{2}|A_{\kappa,1}|$.
    Thus, for $r''\geq r'$ with $\theta\geq f(r'')\geq \frac{\pi}{2}\geq f(r')$:
    \begin{align*}
        \frac{\Per(E;A_{\kappa,1})}{(d-1)\omega_{d-1}}&\geq \int\limits_{r'}^{r''} r^{d-2}\sin^{d-2}f(r)\sqrt{1+r^2f'^2(r)}\dd r\\
        &\geq \kappa^{d-2}\sin^{d-2}\theta\int\limits_{r'}^{r''} rf'(r)\dd r\\
        &\geq \kappa^{d-1}\sin^{d-2}\theta(f(r'')-f(r')).
    \end{align*}
    Similarly, with the argument of \cite{bogelein2017quantitative} for $\theta\geq f^+(r)\geq f^-(r)\geq \frac{\pi}{2}$:
    \begin{align*}
        r^{d-1}\int\limits_{f^-(r)}^{f^+(r)}\sin^{d-2}(\sigma)\dd \sigma&\geq \frac{\kappa^{d-1}}{2}\left(1-\frac{f^-(r)}{\pi}\right)^{d-3}(f^+(r)-f^-(r))\\
        &\geq \frac{\kappa^{d-1}}{2}\left(1-\frac{\theta}{\pi}\right)^{d-3}(f^+(r)-f^-(r)).
    \end{align*}
    Combining both estimates, we see that if there is $r$ with $f(r)\geq\theta$, then there exists a $c>0$ independent of $\kappa$ for $\kappa>\kappa_0$ s.t.\ 
    $$\frac{\Per(E;A_{\kappa,1})}{(d-1)\omega_{d-1}}\geq c (\theta-\frac{\pi}{2}).$$
    Thus, the isoperimetric type inequality holds true in this case.

    Now, we can assume that $f(r)<\theta$. Then, we either have $f(r)\leq \frac{\pi}{2}$ and $$\sin^{d-2}(\sigma)\leq \sin^{d-2}(f(r))$$ for $\sigma\leq f(r)$, or $\frac{\pi}{2}<f(r)<\theta$ which implies 
    $$\frac{\sin^{d-2}(f(r))}{\sin^{d-2}(\theta)}\geq 1.$$ 
    Therefore,
    $$\sin^{d-2}(\sigma)\leq \sin^{d-2}(f(r))\frac{1}{\sin^{d-2}(\theta)}.$$
    We can use a trivial estimate to bound the perimeter
    $$\frac{\Per(E;A_{\kappa,1})}{(d-1)\omega_{d-1}}\geq \int\limits_\kappa^1 r^{d-2}\sin^{d-2}(f(r))\dd r.$$
    Together with the fact that $r$ is bounded above and below, we get
    \begin{align*}
        \frac{\Per(E;A_{\kappa,1})}{(d-1)\omega_{d-1}}&\geq \int\limits_\kappa^1 r^{d-2}\sin^{d-2}(f(r))\dd r\\
        &\geq \frac{1}{\pi} \int\limits_\kappa^1 r^{d-1}\int\limits_0^{f(r)}\sin^{d-2}(f(r))\dd\sigma\dd r\\
        &\geq \frac{\sin^{d-2}(\theta)}{\pi} \int\limits_\kappa^1 r^{d-1}\int\limits_0^{f(r)}\sin^{d-2}(\sigma)\dd\sigma\dd r= \frac{\sin^{d-2}(\theta)}{\pi}|E|.\qedhere
    \end{align*}
\end{proof}

Having this inequality, we can complete the proof of Proposition~\ref{MinkowskiContent}.

\begin{proof}[Proof of Proposition~\ref{MinkowskiContent} --- Part 2]
    Proposition~\ref{Prop:relIsopAnnuli} yields 
    $$\Per(E;A_1)\geq c|E|$$
    for all $E$ s.t.\ $|E|\in \left(\frac{|A_1|}{10}, \frac{9|A_1|}{10}\right)$. As this bounds the isoperimetric function, we can follow the same argumentation as in the first part of this proof to get $v^{-1}(|E_\delta|)\geq v^{-1}(|E|)+\delta$ for $v(r)=e^{c\cdot r}$, i.e., 
    $$\log(|E_\delta\setminus E|)\geq \log(|E|)+c\delta.$$
    With $1+x\leq e^{x}$, we get for $E$ with $|E|=\frac{1}{2}|A_1|$
    $$\frac{|E_\delta\setminus E|}{|A_1|}\gtrsim \delta.$$
    The constant $\frac{1}{2ce^{c|E|}}$ does not vanish as $\kappa\to 1$.
\end{proof}

\begin{remark}
    The usual scaling with the exponent $\frac{n-1}{n}$ on the RHS in the statement of Proposition~\ref{Prop:relIsopAnnuli} for suitable $n\in\mathbb{N}$ (not even $n=d$) cannot be expected as one would expect for the optimal halving set 
    \begin{align*}
        \Per(E;A)&=\omega_{d-1}(1-\kappa^{d-1})\\
        \text{and }|E|&=\frac{1}{2}\omega_d(1-\kappa^d).
    \end{align*}
    Thus, 
    $$\frac{\Per(E;A)}{|E|}\to 2\frac{(d-1)\omega_{d-1}}{d\omega_d}>0,$$
    but both quantities converge linearly (in $1-\kappa$) to zero which prohibits better exponents.
\end{remark}

Now, we can use these bounds on the outer Minkowski content as an estimate for the isoperimetric functional defined in \cite{ledoux2001concentration}. This bound yields our result for $\tau$ using \cite[Proposition 2.1]{ledoux2001concentration}.
 
\begin{proof}[Proof of Lemma~\ref{tauEstimate}]
First, we observe that the modulus of continuity is preserved for our scheme. Namely, $u(x,t)$ is bound by the same modulus as $g$, i.e., $\omega_u\leq \omega_g$.

As in Proposition~\ref{Prop:Monotonicity}, our scheme satisfies for $u\leq v$:
\begin{align*}
    G_h u\leq G_h v \tag{Monotonicity}
\end{align*}
and 
\begin{align*}
    G_h(au+c)=a\cdot G_hu+c \tag{Relabeling}
\end{align*}
as well as that the continuous evolution is homogeneous in space.

Thus, if $|g(x)-g(y)|\leq \omega(|x-y|)$ or equivalently for $\xi\in\mathbb{R}^d$
$$g(x)-g(x+\xi)\leq \omega(|\xi|),$$
we have 
$$G_h(g(x)-g(x+\xi))=u^1(x)-u^1(x+\xi)\leq \omega(|\xi|).$$ 
This implies $\omega_u\leq \omega_g$.

Now, we can reformulate the equation for $\tau$ with the cumulative distribution function in a geometric way:
$$\tau_2=\frac{|\{u\in [m_\mathbb{P},m_\mathbb{P}+\varepsilon)\}|}{|A_r|}$$
or equivalently with $\varphi(x)=\frac{u(rx+x_0)-u(x_0)}{r^2}$, recall $\widetilde m_\mathbb{P}=\frac{m_\mathbb{P}-u(x_0)}{r^2}, \widetilde\varepsilon=\frac{\varepsilon}{r^2}\in o_r(1)$:
$$\tau_2=\frac{|\{\varphi(x)\in[\widetilde m_\mathbb{P},\widetilde m_\mathbb{P}+\widetilde\varepsilon)\}|}{|A_1|}.$$ 
With $E\coloneqq \{\varphi>\widetilde m_\mathbb{P}\}$ ($|E|=\frac{1}{2}|A_1|$) and $E_r\coloneqq E+B_r\cap A_1=\{x\in A_1: \dist(x, E)<r\}$ this can be reformulated as
$$\tau_2\geq \frac{|E_\delta\setminus E|}{|A_1|}$$
where $\delta$ is given by 
$$\frac{\omega_g^{-1}(\varepsilon)}{r}\leq \frac{\omega_u^{-1}(\varepsilon)}{r}=\omega_\varphi^{-1}(\widetilde \varepsilon).$$
This can be seen by the definition as for $x$ s.t.\ $\varphi(x)=\widetilde m_\mathbb{P}$ and $y\in E_\delta\setminus E$ we have 
$$|\varphi(y)-\widetilde m_\mathbb{P}|=|\varphi(y)-\varphi(x)|\leq \omega_\varepsilon(|x-y|)\leq \omega_\varepsilon(\delta)\leq \widetilde\varepsilon.$$
Therefore, with our bound of the Minkowski content by Proposition~\ref{MinkowskiContent}, we achieve the desired bound.
\end{proof}

We will compute the probability of convergence in multiple steps using a union bound in the neighborhoods.

If the initial datum $g$ is bounded uniformly continuous, we can find for any $h>0$ small enough an intensity of the point process s.t.\ the continuous and discrete evolutions are uniformly close with high probability.

\begin{lemma}[Convergence of medians]
    \label{Thm:ConvMedians}
    Let $g\in \BUC(D;\mathbb{R})$ be given. For any $\delta>0$, there exists a sequence of intensities $\Lambda(r, \omega_g)<\infty$ s.t.\ with probability $\mathbb{P}_N>1-h^\frac{\delta}{8}$ we have:
    $$\norm{Q_{h,t}^{\mathbb{P}}g-Q_{h,t}^{N}g}_{L^\infty}\leq h^\frac{\delta}{8} T\norm{g}_{L^\infty}.$$
    
    For $g\in \Lip(D;\mathbb{R})$, we have an explicit version. Let the intensity of the point process be $\Lambda=2\frac{d+5}{k^2}r^{-d-2}\log^2\frac{1}{r}$ with $\frac{1}{k}\coloneqq \Lip(g)$. Then, with probability $\mathbb{P}_N>1-2CT\frac{d+5}{k^2}r\log^2\frac{1}{r} $:
    $$\norm{Q_{h,t}^{\mathbb{P}}g-Q_{h,t}^{N}g}_{L^\infty}\leq \frac{1}{\sqrt{\frac{1}{2}\log \frac{1}{h}}} T\norm{g}_{L^\infty}.$$
\end{lemma}

Therefore, we can conclude the uniform convergence of our non-local discrete scheme to the viscosity solution of mean curvature flow using Lemma~\ref{UniformConvergence} and Lemma~\ref{Thm:ConvMedians}.

\begin{thm}
\label{Cor:Conv}
    Let $g\in BUC(D;\mathbb{R})$ the initial datum be given. Then, there is a sequence $r(N)\to 0$ as $N\to \infty$ s.t.\ the discrete evolution converges a.s.\ uniformly to the time rescaled viscosity solution to mean curvature flow as $N\to \infty, r(N)\to 0$
    $$Q_{h,t}^{N}g\to u(x,t)\text{ in } L^\infty(D\times [0,T]).$$
    If $g\in \Lip(D)$, this assumption can be rewritten as $r\gg \frac{\log N}{N^\frac{1}{d+2}}$.
\end{thm}

\begin{proof}[Proof of Lemma~\ref{Thm:ConvMedians}]
    In the following let $\tau$ be a lower bound on $\tau_{1/2}$.
    First, we fix a time slice $t$ and examine the probability that one step behaves in the way we want:
    \begin{align*}
        \mathbb{P}_N&\left(\bigcap\limits_{x\in P}\{|m_\mathbb{P}u(x)-m_Nu(x)|<h\widetilde\varepsilon\}\right)\\
        &\geq 1-\mathbb{P}_N\left(\bigcup\limits_{x\in P}\{|m_\mathbb{P}u(x)-m_Nu(x)|>h\widetilde\varepsilon\}\right)\\
        &\geq 1-\sum\limits_{x\in P}\mathbb{P}_N\left(\{|m_\mathbb{P}u(x)-m_Nu(x)|>h\widetilde\varepsilon\}\right).
    \end{align*}
    To bound this sum, we will combine the local estimate from above
    \begin{align*}
        \mathbb{P}_N\left(|m_\mathbb{P}u-m_nu|>h\widetilde\varepsilon\left|\#(P\cap B_r(x))=n+1\right.\right)&\lesssim e^{-\tau^2 n}
    \end{align*}
    with a counting argument, given that the total number of points $N$ is known.
    \begin{align*}
        \mathbb{P}_N\left(|m_\mathbb{P}u-m_nu|>h\widetilde\varepsilon\left|N\right.\right)&\lesssim \sum\limits_{n=0}^{N-1} e^{-\tau^2 n}\mathbb{P}_N(\#(P\cap B_r(x))=n+1)\\
        &=\sum\limits_{n=0}^{N-1} e^{-\tau^2 n}\binom{N-1}{n}(r^d)^n(1-r^d)^{N-1-n}\\
        &=\left(1-r^d\left(1-e^{-\tau^2}\right)\right)^{N-1}\\
        &\leq e^{-r^d\left(1-e^{-\tau^2}\right)(N-1)}\\
        &\leq e^{-(N-1)r^d\frac{\tau^2}{2}}.
    \end{align*}
    Here, we used $1-x\leq e^{-x}$ and $1-e^{-x}\geq \frac{x}{2}$ for $0\leq x\leq 1$.

    Finally, we can combine this with the Poisson-Point-Process which has a random number of points sampled.
    \begin{align*}
        \mathbb{P}_N&\left(\bigcup\limits_{x\in P}\{|m_\mathbb{P}u(x)-m_Nu(x)|>h\widetilde\varepsilon\}\right)\\
        &=\sum\limits_{N=0}^\infty \mathbb{P}_N\left(\left.\bigcup\limits_{x\in P}\{|m_\mathbb{P}u(x)-m_Nu(x)|>h\widetilde\varepsilon\}\right| N\right)\frac{\Lambda^N}{N!}e^{-\Lambda}\\
        &\lesssim \sum\limits_{N=0}^\infty \frac{Ne^{-(N-1)r^d\frac{\tau^2}{2}}\Lambda^N}{N!}e^{-\Lambda}\\
        &=\Lambda \sum\limits_{N=0}^\infty \frac{e^{-Nr^d\frac{\tau^2}{2}}\Lambda^N}{N!}e^{-\Lambda}\\
        &=\Lambda e^{-\left(1-e^{r^d\frac{\tau^2}{2}}\right)\Lambda}\\
        &\leq \Lambda e^{-r^d\frac{\tau^2}{2}\Lambda}.
    \end{align*}
    This can be combined with a union bound in time under the observation that we have $\frac{T}{h}$ time steps of length $h$:
    \begin{align*}
        \mathbb{P}_N\left(\bigcup\limits_{x,t}\{|m_\mathbb{P}u(x)-m_Nu(x)|>h\widetilde\varepsilon\}\right)&\lesssim \frac{T}{h}\Lambda e^{-r^d\frac{\tau^2}{2}\Lambda}.
    \end{align*}
    With $\tau=\frac{\omega_g^{-1}(\widetilde\varepsilon h)}{\sqrt{h}}$ and $\Lambda=(c(r)+d)\frac{2r^{-(d-2)}}{\left(\omega_g^{-1}(\widetilde\varepsilon r^2)\right)^2}\log \frac{1}{r}$ this yields the probability:
    \begin{align*}
        \mathbb{P}_N\left(\bigcup\limits_{x,t}\{|m_\mathbb{P}u(x)-m_Nu(x)|>h\widetilde\varepsilon\}\right)&\lesssim T\frac{2(c+d)}{\left(\omega_g^{-1}(\widetilde\varepsilon r^2)\right)^2}\log\frac{1}{r}r^c\\
        &=2T(c+d)\frac{r^c}{(\omega_g^{-1}(r^{2+\frac{\delta}{4}}))^2}.
    \end{align*}
    In the last line, we plugged in $\widetilde\varepsilon=h^\frac{\delta}{8}$. Observe that $(c+d)r^c\to 0$ as $c\to \infty$. Thus, we can find a $c(r)$ for each $r>0$ s.t.\ 
    $$2T(c+d)\frac{r^c}{(\omega_g^{-1}(r^{2+\frac{\delta}{4}}))^2}\leq r^\frac{\delta}{4}$$
    as $r\to 0$.

    For $g$ Lipschitz, we have the bound $\tau=k\widetilde\varepsilon r$ and with the explicit choice of $\Lambda=2\frac{d+5}{k^2}r^{-d-2}\log^2\frac{1}{r}, \widetilde\varepsilon=\frac{1}{\sqrt{\log\frac{1}{r}}}$:
    \begin{align*}
        \mathbb{P}_N\left(\bigcup\limits_{x,t}\{|m_\mathbb{P}u(x)-m_Nu(x)|>h\widetilde\varepsilon\}\right)&\lesssim \frac{2T}{r^2}\frac{d+5}{k^2}r^{-d-2}\log^2\frac{1}{r} e^{(d+5)\log r}\\
        &=2T\frac{d+5}{k^2}r\log^2\frac{1}{r} .
    \end{align*}
    This bound goes to zero as $r\to 0$.

    To lift this result to the evolution we have to iterate the error estimate:
    \begin{align*}
        \norm{\left(G_h^{\mathbb{P}}\right)^ng-\left(G_h^N\right)^ng}_{L^\infty}&\leq \norm{(m_\mathbb{P}-m_N)\left(G_h^{\mathbb{P}}\right)^{n-1}g}_{L^\infty}\\
        &\quad+\norm{m_N\left(\left(G_h^{\mathbb{P}}\right)^{n-1}g-\left(G_h^N\right)^{n-1}g\right)}_{L^\infty}\\
        &\leq h\widetilde\varepsilon\norm{g}_{L^\infty}+\norm{\left(G_h^{\mathbb{P}}\right)^{n-1}g-\left(G_h^N\right)^{n-1}g}_{L^\infty}\\
        &\leq n h\widetilde\varepsilon\norm{g}_{L^\infty}\leq T\widetilde\varepsilon\norm{g}_{L^\infty}.\qedhere
    \end{align*}
\end{proof}

\section{\texorpdfstring{${\TL}^2$-}{TL2-}convergence \texorpdfstring{---}{-} Heat Flow}
\label{TL-convergence}

In this section, we will prove that the discrete sampled heat flow converges in the mean field limit to its continuous counterpart. Here, we show ${\TL}^2$-convergence in the setting of logarithmically many points in the local neighborhoods. This is the optimal regime as discussed in Remark~\ref{Rem:Optimality-r}. It should be noted that the same techniques as in the previous sections can be used to obtain $L^\infty$-convergence with quadratically (for Lipschitz functions) many points. The precise statement can be found in Theorem~\ref{Thm:TLConv} and Corollary~\ref{Cor:HeatFlows}. For more general settings and extensions, see Corollary~\ref{Cor:generalSettingTL}. In these, we extend our proof to other domains and onto manifolds. In addition, we consider densities, different kernels and other sampling processes. Another consequence is a separated weak convergence result for the MBO scheme which is written down in Corollary~\ref{Cor:MBOScheme}. This result extends~\cite[Corollary 1]{jonatim} to the optimal regime.

Let again $D\coloneqq\mathbb{T}^d$ be the $d$-dimensional flat torus. Additionally, let $P_N$ denote a realization of the point process with $N$ uniformly distributed points in $D$.

To speak about a $L^2$-type convergence of $u_N\to u$, we need to compare a function on the continuum to a discrete realization. This can be done using transport plans which then compare the values closest (in a graph-Wasserstein-$2$-distance) to the discrete samples.

To this end, we defined the ${\TL}^2$-convergence in Definition~\ref{Def:TL2} and Definition~\ref{Def:TL2Conv}. 

We want to show that the solution to the heat equation on a discrete space converges to the continuous heat flow as long as the kernel size is large compared to the connectivity radius which depends on $N$. 

To give the precise setting, we will examine the implicit Euler scheme with a non-normalized but averaged Laplacian (the random walk Laplacian). Here, we normalize our kernel by the approximated number of points in a ball which is given by their volume fraction multiplied with the entire number of points $Nr^d\omega_d$.
Then, the implicit Euler scheme of the average filter can be written as
$$\frac{u_N^{n+1}(x)-u_N^n(x)}{\tau}=\frac{1}{Nr^{d+2}\omega_d}\left(\sum\limits_{y\in B_r(x)}u_N^{n+1}(y)-u_N^{n+1}(x)\right).$$
This can be reformulated as an minimizing movement:
\begin{align}
    \label{Eq:HE-MM}
    u_N^{n+1}=\argmin\limits_u\left\{E_{r,N}(u)+\frac{1}{2\tau}d_N^2(u,u_N^n)\right\}.
\end{align}
Here, the energy is a discrete Dirichlet energy
\begin{align*}
    E_{r,N}(u)&\coloneqq \frac{1}{2N}\sum\limits_{x\in P_N}\frac{1}{Nr^d\omega_d}\sum\limits_{y\in B_r(x)}\left|\frac{u(x)-u(y)}{r}\right|^2
\end{align*}
and the distance takes the form of a discrete $L^2$-distance
\begin{align*}
    d_N^2(u,u_N^n)&=\frac{1}{N}\sum\limits_{x\in P_N}|u(x)-u_N^n(x)|^2.
\end{align*}

To speak about the ${\TL}^2$-convergence, we will use the continuous measure
$$\mu\coloneqq \mathcal{L}^d\llcorner_{D}$$ which is the Lebesgue measure and a discrete approximation adapted to our point process $P_N$
$$\mu_N\coloneqq\frac{1}{N}\sum\limits_{x\in P_N}\delta_{x}.$$

In analogy to Section~\ref{LinfConv}, we define the update $G^{r, N}_{\tau} g\coloneqq u^1_N$ where $u^n_N$ is the evolution given $g$ as the initial datum. Additionally, let $Q^{r, N}_{\tau, t}g\coloneqq \left(G^{r,N}_\tau\right)^j g$ for $t\in [j\tau, (j+1)\tau)$ with $j\in\mathbb{N}_{\geq 0}$. 

For the proof of the convergence, we will employ a $\Gamma$-convergence of the associated energies which is done in the following lemma. This lemma can be found as \cite[Theorem 12]{jonatim} which extends \cite[Theorem 1.4]{SlepcevTrillos2} to the manifold setting.

\begin{lemma}[$\Gamma$-convergence and compactness of Dirichlet energies, {\cite[Theorem 12]{jonatim}}]
    \label{Lemma:GammaConv}
    Given $\mu_N, \mu$ and the energies $E_{r,N}, E$ as before. Then, there exists a constant $C>0$ s.t.\ for $r\gg \frac{\log^\frac{1}{d} N}{N^\frac{1}{d}}$ the energies converge a.s.\ 
    $$E_{r,N}\overset{\Gamma-{\TL}^2}{\longrightarrow}k_2\cdot E$$
    with a constant $k_2$ given by $k_2\coloneqq\fint\limits_{B_1(0)}x_1^2\dd x$.
    
    Moreover, let $u_N\in L^2(D,\mu_N)$ be s.t.\ 
    \begin{align*}
        \sup_N E_{r(N), N}(u_N)&<\infty,\\
        \sup_N \norm{u_N}_{L^2(D,\mu_N)}&<\infty.
    \end{align*}
    Then, the sequence $(u_N)_{N\in\mathbb{N}}$ is precompact in ${\TL}^2(D)$, i.e., the energies satisfy a $\Gamma$-compactness.
\end{lemma}

To be rigorous, we can only talk about ${\TL}^2$-convergence after extending the energies to the whole space ${\TL}^2(D)$. This is done in the following way:
\begin{align*}
    E_{r,N}(u,\nu)&\coloneqq\begin{cases}\frac{1}{2}\int\limits_D \frac{1}{ \omega_d r^d}\int\limits_{B_r(x)}\left|\frac{u(x)-u(y)}{r}\right|^2\dd\nu(y)\dd\nu(x) &\text{if } \nu=\mu_N, \\
    \infty &\text{else},\end{cases}\\
    E(u,\nu)&\coloneqq\begin{cases}\frac{1}{2}\int\limits_D |\nabla u|^2\dd \nu(x) &\text{for } \nu=\mu, u\in H^1(D,\mu),\\
    \infty &\text{else}.\end{cases}
\end{align*}
Here, the factor $r^d\omega_d$ again represents the volume of the ball to make the integral in the limit a mean integral.

Now, we are in the situation to state and prove our main theorem of this section.

\begin{thm}[Convergence of minimizing movement for heat flows]
\label{Thm:TLConv}
    Let $u_N^n(x)$ be defined as the heat flow minimizing movement in \eqref{Eq:HE-MM} with initial datum $g_N\in L^2(D,\mu_N)$ for $\tau>0$ and a discrete point sample $P_N$ with associated measure $\mu_N$. Additionally, let $u$ denote the solution to the $\mu$-heat equation on $D$ starting from $g\in L^2(D,\mu)$. Moreover, assume that
    \begin{align*}
        g_N\to g \text{ in }{\TL}^2(D),\\
        \limsup\limits_{N\to \infty} E_{r,N}(g_N)<\infty.
    \end{align*}    
    Then, given that $r\gg \frac{\log^\frac{1}{d} N}{N^\frac{1}{d}}$, we have a.s.\ that for all $t>0$:
    $$(\mu_N, Q_{\tau,t}^{r,N}g_N)\to (\mu, u(\cdot, t))\text{ in }{\TL}^2(D)$$
    as $\tau\to 0$ and $N\to\infty$.
\end{thm}

This scaling of $r$ corresponds to logarithmically many points in the local neighborhoods.

\begin{remark}
\label{Rem:Optimality-r}
    The condition $r\gg \frac{\log^\frac{1}{d} N}{N^\frac{1}{d}}$ is a weakened version of $r\gg d_\infty(\mu,\mu_N)$ which can be bounded by $\begin{cases}
        \frac{\log^\frac{3}{4} N}{N^\frac{1}{2}}&\text{for }d=2,\\\frac{\log^\frac{1}{d} N}{N^\frac{1}{d}} &\text{for }d\geq 3.
    \end{cases}$
    
    Here, $d_\infty$ denotes the $\infty$-Wasserstein distance.
    This improvement is due to \cite{calder2022improved}. The weaker bound is optimal as $\frac{\log^\frac{1}{d} N}{N^\frac{1}{d}}$ coincides with the connectivity radius. Below this bound, there is a non-vanishing probability that the discrete graph built by $P_N$ with the neighborhoods given by the kernel is not connected anymore. This would for instance contradict the compactness property.
\end{remark}

\begin{remark}
\label{Rem:Admiss-initial}
    For a recovery sequence $g_N\in L^2(D,\mu_N)$ with $g_N\to g\in {\TL}^2(D)$, the assumptions in the theorem are satisfied. Moreover, for a constant sequence of smooth functions $g_n\equiv g\in C^\infty(D)$, we have by \cite[Theorem~12]{jonatim}
    \begin{align*}
        \sup\limits_N E_{r(N),N}(g_N)<\infty,\\
        \sup\limits_N\norm{g_N}_{L^2(D,\mu_N)}<\infty.
    \end{align*}
    Thus, these sequences are admissible.
\end{remark}

\begin{thm}[Convergence of heat flows]
    \label{Cor:HeatFlows}
    Let $u_N$ be the heat flow starting at $g_N\in L^2(D,\mu_N)$ associated with the discrete Laplacian on $P_N$ and $u$ be the continuous heat flow (with time rescaling by $k_2$) with initial datum $g\in L^2(D,\mu)$ s.t.\ $g_N\circ T_N\rightharpoonup g$ in $L^2(D,\mu)$. Then a.s.\ for every $t>0$ they converge
    $$u_N(\cdot, t)\to u(\cdot, t) \text{ in }{\TL}^2(D).$$
\end{thm}

This weak convergence of $g_N$ is sometimes called weak ${\TL}^2$-convergence, cf.~\cite[Definition 1]{jonatim}.

Here, we defined $T_N$ to be transport maps between $\mu$ and $\mu_N$ that satisfy
$$\begin{cases}
    \limsup\limits_N \frac{N^\frac{1}{2}}{\log^\frac{3}{4} N}\norm{T_N-\text{Id}}_\infty\leq C&\text{for }d=2,\\ 
    \limsup\limits_N \frac{N^\frac{1}{d}}{\log^\frac{1}{d} N}\norm{T_N-\text{Id}}_\infty\leq C&\text{else}.
\end{cases}$$
These maps exist by \cite[Theorem 9]{jonatim}.

\begin{remark*}
    The uniform energy bound is not needed in the setting of Theorem~\ref{Cor:HeatFlows}.
\end{remark*}

The strategy of the proof will be to use the $\Gamma$-convergence of the energies combined with the fact that the distances are continuous in ${\TL}^2(D)$. This yields the convergence of the minimizers and thus of the evolution. Afterwards, the fact that the energies are convex, gives the convergence independent of the regime $N\to\infty$ and $\tau,r(N)\to 0$. This enables us to take $\tau\to 0$ to get the convergence of the discrete heat flow to the continuous flow.

\begin{proof}[Proof of Theorem~\ref{Thm:TLConv}]
    We will use as an intermediate step, the space continuous minimizing movement $u^n$ of $k_2E$ with $L^2(D,\mu)$-distance $d$.
    
    In the same fashion as we extended the energies to ${\TL}^2(D)$, we note that the distances are restricted distances of the ${\TL}^2$-distance:
    \begin{align*}
        d_N(u,v)&= d_{{\TL}^2}((\mu_N,u), (\mu_N,v))=\norm{u-v}_{L^2(D,\mu_N)},\\
        d(u,v)&\coloneqq \norm{u-v}_{L^2(D,\mu)}.
    \end{align*}
    
    We will examine the implicit Euler scheme for the heat equation:
    \begin{align*}
        u^{n+1}&=\argmin\limits_u\left\{k_2E(u)+\frac{1}{2\tau}d^2(u,u^n)\right\}
    \end{align*}
    as well as the space discretization
    \begin{align*}
        u_N^{n+1}&=\argmin\limits_u\left\{E_{r(N),N}(u)+\frac{1}{2\tau}d_N^2(u,u_N^n)\right\}.
    \end{align*}

    By Lemma~\ref{Lemma:GammaConv}, we know that $E_{r,N}\overset{\Gamma-{\TL}^2}{\longrightarrow}k_2E$ in our regime $r(N)\gg \frac{\log^\frac{1}{d} N}{N^\frac{1}{d}}$.
    Moreover, we have that $d_N$ and $d$ are both $d_{{\TL}^2}$ (restricted on their second argument and on their respective measures) when extended to ${\TL}^2(D)$. Thus, $d_N^2(\cdot, u_N^n)$ converges continuously to $d^2(u,u^n)$ if $u_N^n\to u^n$ in ${\TL}^2(D)$, i.e., for $(\mu_N,u_N)\to (\mu,u)$ in ${\TL}^2(D)$, we have
    $$d_{{\TL}^2}^2((\mu_N,u_N), (\mu_N,u_N^n))\to d_{{\TL}^2}^2((\mu,u), (\mu,u^n)).$$
    
    To see this, we compute with the triangle inequality
    \begin{align*}
        |&d_N(u_N, u_N^n)-d(u, u^n)|\\
        &=|d_{{\TL}^2}((\mu_N,u_N), (\mu_N,u_N^n))-d_{{\TL}^2}((\mu,u), (\mu,u^n))|\\
        &\leq |d_{{\TL}^2}((\mu_N,u_N), (\mu,u))+d_{{\TL}^2}((\mu,u), (\mu_N,u_N^n))-d_{{\TL}^2}((\mu,u), (\mu,u^n))|\\
        &= d_{{\TL}^2}((\mu_N,u_N), (\mu,u))+d_{{\TL}^2}((\mu,u^n), (\mu_N,u_N^n))\to 0.
    \end{align*}

    Since the energies $E_{r,N}$ (and thus $E$) are convex, equi-coercive, positive and lower semi-continuous, we can use the theory of \cite{braides2014local} to get convergence of the minimizing movement independent of the regime $N\to\infty$ and $\tau,r(N)\to 0$. Note that by convexity, minimizers are unique.
    For each fixed $r(N), N$, we can use the fundamental theorem of $\Gamma$-convergence. 
    By the $\Gamma$-convergence combined with compactness, we have that global minima of the minimizing movement converge (up to a subsequence) to global minima of the limits. This means that for each $\tau>0$ and with converging $u_N^n$, in $N$, the minimizers $u_N^{n+1}$ for $N<\infty, r>0$ converge (up to a subsequence) to the minimizer of the continuum limit as $N\to \infty$ in the strong ${\TL}^2$-topology.
    
    Thus, the evolutions $u_N^n$ converge in $N$ as long as the initial data $g_N$ converge. This can be iterated for fixed $\tau>0$. 
    Having the convergence, we also get, cf.\ \cite[Chapter 11]{braides2014local}, that for each $N$ (and the limit) there is a unique limit $u_N(t)$ of the minimizing movement as $\tau\to 0$ and that
    $$\sup_n d_N(u_N,u_N^n)\leq 6\sqrt{\tau}\sqrt{E_{r(N),N}(g_N)}.$$
    Hence, we can compute for some $t>0$:
    \begin{align*}
        d_{{\TL}^2}((\mu_N,u_N),(\mu,u))&\leq d_N(u_N^n, u_N)+d(u,u^n)+d_{{\TL}^2}((\mu_N,u_N^n),(\mu,u^n))\\
        &\lesssim \sqrt{\tau}+\sqrt{\tau}+o_N(1).
    \end{align*}
    The constant depends only on $\sup E_{r(N),N}(g_N)<\infty$ which is bounded by assumption. The same is true for $E(g)$ by $\Gamma$-convergence and the liminf-inequality as $g_N\to g$ in ${\TL}^2(D)$.

    Now, $\tau\to 0$ and $N\to\infty$ yield the convergence of the minimizing movements to the heat flow. We obtain the result for the whole sequence by the subsequence lemma. (I.e., first we take an arbitrary subsequence and get convergence of a further subsequence. Since all limits are the same, we conclude the convergence of the whole sequence.)
    
    We also get the same estimate for the minimizing movement compared to the limit:
    \begin{align*}
        d_{{\TL}^2}((\mu_N,u_N^n),(\mu,u))&\lesssim \sqrt{\tau}+o_N(1).
    \end{align*}
    The constant again only depends on $\sup_N E_{r(N),N}(g_N)$.
\end{proof}

For Theorem~\ref{Cor:HeatFlows}, we have to show that the assumptions on the initial data can be weakened. 
Now, we can use Theorem~\ref{Thm:TLConv} to get the convergence result $u_N(t)\to u(t)$ in ${\TL}^2(D)$ for all positive times $t>0$ and initial data $g_N\rightharpoonup g$. This proof uses the techniques of \cite[Proof of Theorem 2]{jonatim} and we give an overview of them for the convenience of the reader.

\begin{proof}[Proof of Theorem~\ref{Cor:HeatFlows}]
    Fix $t>0$ and let $u_N, u$ be the heat flow evolutions to the given initial data $g_N\rightharpoonup g$. They can be obtained as limiting minimizing movements as $\tau\to 0$.
    Their energies are bounded in time via
    $$\int\limits_0^t E_{r(N),N}(u_N(s))\dd s\lesssim \norm{g_N}_{L^2(D,\mu_N)}^2.$$
    Taking the $\liminf$ in $N$ and using Fatou's lemma, this yields
    $$\int\limits_0^t \liminf\limits_{N\to \infty} E_{r(N),N}(u_N(s))\dd s\lesssim C((g_N)_{N\in\mathbb{N}})$$
    and thus a uniform bound.
    By the mean-value theorem (and via going to a subsequence without relabeling), this gives us an $s>0$ s.t.
    $$\lim\limits_N E_{r(N),N}(u_N(s))\lesssim C_g.$$
    Using the monotonicity in time of the energies, we get for $t$:
    $$\sup\limits_N E_{r(N),N}(u_N(t))<\infty.$$
    Now, the compactness from Lemma~\ref{Lemma:GammaConv} gives us a subsequence s.t.\ $u_N(t)\to v(t)$ in ${\TL}^2(D)$ for a function $v\in L^2(D,\mu)$. To see that $v$ coincides with $u$, we test against test functions $\varphi\in C^\infty(D)$. By the ${\TL}^2$-convergence, we have $u_N\circ T_N\to v$ in $L^2(D,\mu)$ and thus
    \begin{align*}
        \int\limits_D v(t)\varphi\dd\mu&=\lim\limits_{N\to \infty}\int\limits_D u_N(t)\circ T_N \varphi \dd\mu\\
        &=\lim\limits_{N\to \infty}\int\limits_D u_N(t)\circ T_N \varphi\circ T_N\dd\mu.
    \end{align*}
    Here, we used that $T_N\to \text{Id}$ in $L^\infty(D)$. This implies $\varphi\circ T_N\to \varphi$ uniformly and thus also in $L^2(D,\mu)$. Using the self-adjointness of the heat semigroup (which can be seen for instance by the Dirichlet form definition), we get
    \begin{align*}
        \int\limits_D v(t)\varphi\dd\mu&=\lim\limits_{N\to \infty}\int\limits_D g_N\varphi_N(t)\dd\mu_N\\
        &=\lim\limits_{N\to \infty}\int\limits_D g_N\circ T_N\varphi_N(t)\circ T_N\dd\mu\\
        &=\int\limits_D g\varphi(t)\dd\mu\\
        &=\int\limits_D u(t)\varphi\dd\mu.
    \end{align*}
    For these steps, we observe the strong ${\TL}^2$-convergence of $\varphi_N(t)\to \varphi(t)$ which is an application of Theorem~\ref{Thm:TLConv} with initial data $\varphi_N\equiv \varphi\in C^\infty(D)$. This initial data is admissible by Remark~\ref{Rem:Admiss-initial}. Together with the weak ${\TL}^2$-convergence of $g_N\rightharpoonup g$, we get the desired limit.
    
    This proves that $v(t)$ coincides with $u(t)$ in $L^2(D)$ and thus $u_N\circ T_N$ converges in $L^2(D,\mu)$ and $u_N$ in ${\TL}^2(D)$ to $u(t)$ for all $t>0$.
\end{proof}

So far, we presented the results of this section in the base setting of the flat torus in $\mathbb{R}^d$ but the same proof works in more generality. We will extend the result to Lipschitz domains and even to smooth manifold with continuous densities. Moreover, we will allow for general kernels and different point processes. Here, the evolution takes homogeneous Neumann boundary conditions but there are ways to allow for different conditions as discussed in Section~\ref{Program}. For the statement, we have to define what kernels are admissible instead of $K(x,y)=\frac{1}{\omega_d r^d}\mathds{1}_{B_r(0)}(x-y)$.

\begin{defi}[Admissible kernel]
\label{Def:AdmissibleKernel}
    We call a kernel $K:\mathbb{R}^d\times\mathbb{R}^d\to \mathbb{R}_{\geq 0}$ admissible, if it is symmetric and has a radial representation, i.e., by abuse of notation $K(x,y)=K(|x-y|)$ for $K:\mathbb{R}_{\geq 0}\to \mathbb{R}$ s.t.\ 
    \begin{itemize}
        \item $K$ is continuous at $0$ and fulfills $K(0)>0$, 
        \item $K$ is non-increasing,
        \item $\int\limits_{0}^\infty K(r)r^{d+1}\dd r<\infty$.
    \end{itemize}
    For such a kernel, we define the local kernel sequence $K_r(|x|)\coloneqq \frac{1}{r^d}K\left(\frac{|x|}{r}\right)$.
\end{defi}
The integral constraint is equivalent to saying that $k_2$ has to be finite.
This definition can be done analogously for kernels on manifolds. In this case, we replace $K(|x-y|)$ by $K(d_M(x,y))$.

\begin{Corollary}
\label{Cor:generalSettingTL}
    Let $M$ be a compact manifold of dimension $d$ with continuous density $\rho>0$ and smooth boundary and let $K$ be an admissible kernel. We set $\mu\coloneqq \rho {\vol}_M$. Moreover, let $P$ be a point process with sampled distribution $\mu$ and let $\mu_N$ be the associated sampled measure to a realization $P_N$.
    
    Let $E_{r,N}$ the discrete Dirichlet energy be defined via 
    $$E_{r, N}(u)\coloneqq \int\limits_M\int\limits_M K_r(x,y)|u(x)-u(y)|^2\dd \mu_N(y)\dd\mu_N(x)$$
    and the limiting weighted Dirichlet energy be
    $$E(u)\coloneqq \int\limits_M|\nabla u|^2\rho^2\dd {\vol}_M$$
    for $u\in H^1(M,\mu)$.

    Moreover, let $u_N$ be the heat flow associated to the Dirichlet form $E_{r,N}$ and $u$ be the continuous heat flow of $k_2 E$ where
    $$k_2\coloneqq \int\limits_{\mathbb{R}^d}K(|x|)x_1^2\dd x.$$
    Then, for $r(N)\gg d_\infty(\mu_N,\mu)$, we have a.s.\ for every $t>0$
    $$u_N(\cdot, t)\to u(\cdot, t) \text{ in }{\TL}^2(M).$$
\end{Corollary}
If one considers the i.i.d.\ sampling from before, we can again strengthen the result to $r(N)\gg\frac{\log^\frac{1}{d}N}{N^\frac{1}{d}}$.

The proof of this corollary is analogous to the previous one since all results of \cite{jonatim} we used also work in this setting, especially the extended $\Gamma$-convergence of the Dirichlet energies of \cite{CalderSlepcev}. Note that their imposed regularity of $K\in C^2(M)$ and $\rho\in C^\infty(M)$ is not necessary for the results we are interested in.

Therefore, we extend \cite[Corollary 1]{jonatim} to the optimal setting of $r\gg \frac{\log^\frac{1}{d}N}{N^\frac{1}{d}}$:
\begin{Corollary}[MBO one step convergence, cf.~{\cite[Corollary 1]{jonatim}}]
\label{Cor:MBOScheme}
    Let $\chi_N: P_N\to \{0,1\}, \chi:M\to \{0,1\}$ be characteristic functions s.t.\ $\chi_N$ converges weakly to $\chi$ in ${\TL}^2(M)$ (see Theorem~\ref{Cor:HeatFlows}). Moreover, let $\chi_N^h, \chi^h$ be the evolution of one step of the MBO scheme (suitable rescaled for the continuum limit) with time step size $h>0$. Then, $\chi_N^h$ converges weakly to $\chi^h$ in ${\TL}^2(M)$.
\end{Corollary}
The proof is the same as in the reference. This corollary also works for the multiphase setting.

\section{Implementation and program variants}
\label{Program}

In this section, we discuss the implementation and variants of the scheme. Additionally, we show a way to extend it to inhomogeneous Neumann conditions with the so-called Young angle condition. Lastly, we conclude by discussing ways to incorporate Dirichlet conditions and how to include the framework of semi-supervised learning.

\subsection{Program}

The program allows to change $N, r$ and $\kappa$ for the kernel $A_{r,\kappa}$ as well as $\alpha$ for the Young angle freely. The points are standard sampled from the unit square but one can plug in other point processes for sampling as well. In Figure~\ref{Fig:Classifier}, examples of dumbbells in the task of classification are listed. Versions for the scheme on characteristic functions as well as for the level set evolution are supported separately.
For visualization purposes, the implementation is in 2d but everything works in higher dimensions as well. The implementations can be found on \url{https://mathrepo.mis.mpg.de/Medianfilter/}.

These encode different versions of the level set approach, a majority rule like version for each level set individually and different kernels. Additionally, comparisons to converging schemes to get an $L^\infty$-error or Hausdorff-distance error respectively, are supported. The latter can be used as an upper bound to the former. 
For Neumann boundary conditions, we use a $k$-select algorithm. The program works with sufficient speed with over 10 million samples.

\begin{figure}[!ht]
    \begin{tblr}{
      colspec = {X[c,h]X[c,h]X[c,h]X[c,h]},
      hlines = {1pt},
      vlines = {1pt},
      width=\linewidth,
    }
      \includegraphics[width=\linewidth]{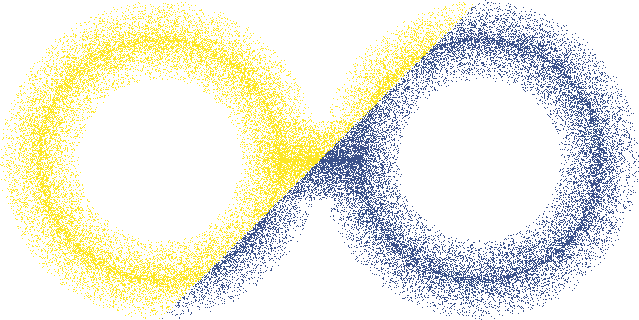} & \includegraphics[width=\linewidth]{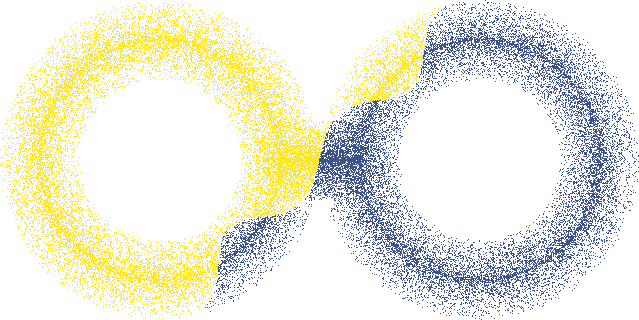} & \includegraphics[width=\linewidth]{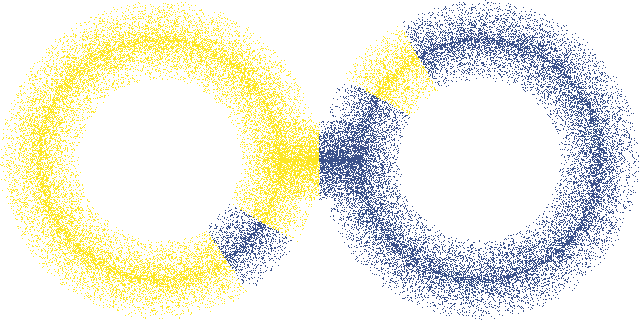} & \includegraphics[width=\linewidth]{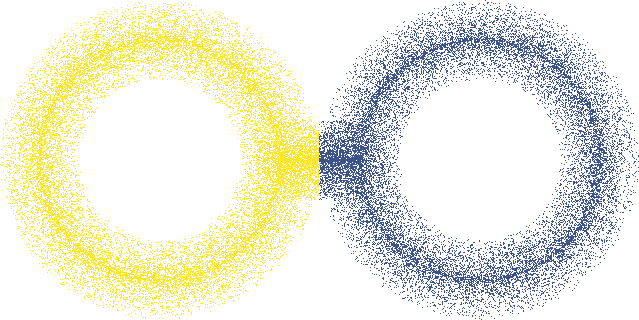} \\
      \includegraphics[width=\linewidth]{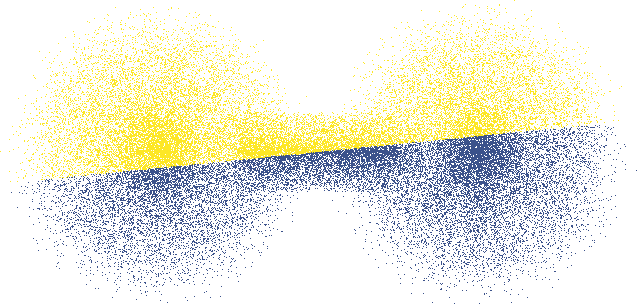} & \includegraphics[width=\linewidth]{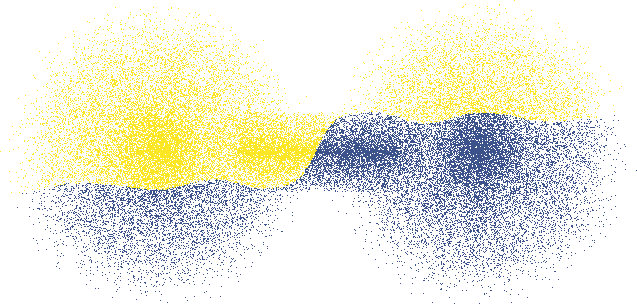} & \includegraphics[width=\linewidth]{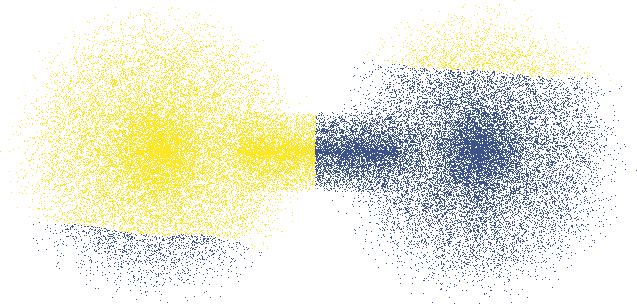} & \includegraphics[width=\linewidth]{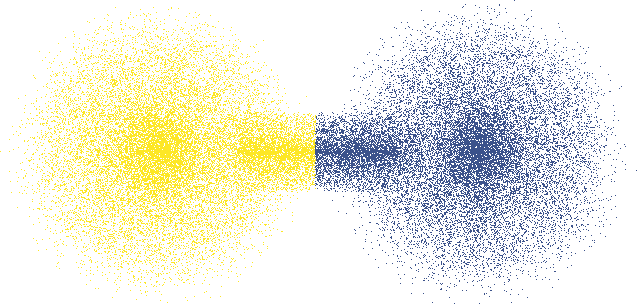} \\
      \includegraphics[width=\linewidth]{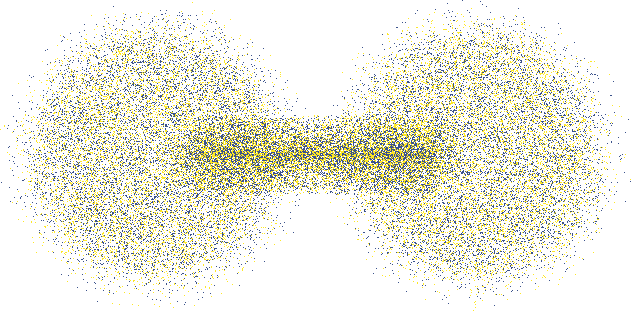} & \includegraphics[width=\linewidth]{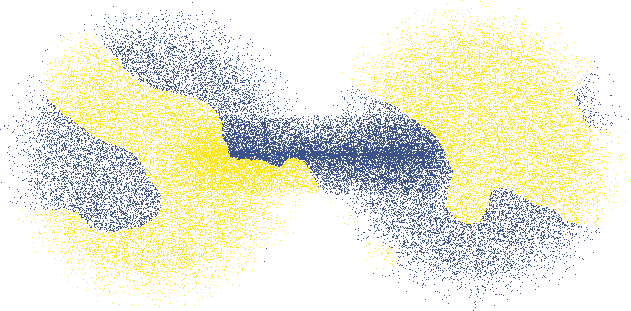} & \includegraphics[width=\linewidth]{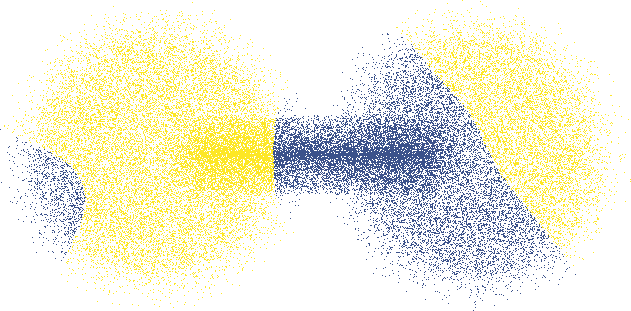} & \includegraphics[width=\linewidth]{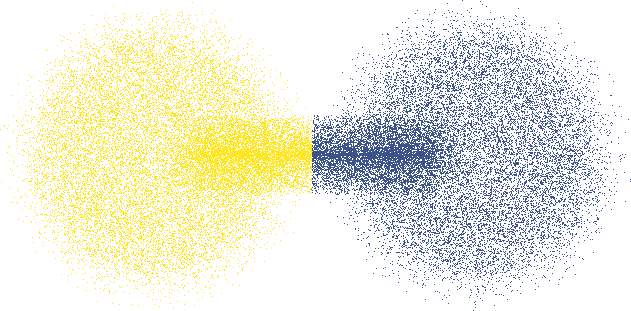} \\
    \end{tblr}
    \caption{Our algorithm applied to three classification problems. From left to right is the evolution over time.}
    \label{Fig:Classifier}
\end{figure}

For a kernel which is not a characteristic function, one can use~\cite[Algorithm 2]{esedoglumedianfilter} to compute the median, i.e., we order the values and increase them until the (in-)equalities defining the median are satisfied.

In the following, we will discuss simplifications which were used in the program and why they do not influence the result but speed up the computations. 
As discussed in Section~\ref{LinfConv}, the only case in which a ball produces a wrong result compared to an annulus is close to the extremal points if the derivative is too small. This is why, in the case of a single level set, we terminate the simulation if the external distance of the levelset becomes too small. Hence, we can use the ball neighborhood and view the change at points near extremal points as a relabeling/new level set function after each time step. The caveat of this is that the algorithm only produces correct results until the level set is too close to an extremal point.

Moreover, in the case of a single level set, we can restrict us to the threshold values $0$ and $1$ (after a relabeling). To see this, assume that we evolve $g$ into $u$ at the threshold $q$. We define the thresholding 
\begin{align*}
    T:\mathbb{R}&\to \{0,1\}\\
    x&\mapsto \mathds{1}_{\{x\geq q\}}.
\end{align*}
It then suffices to see that $T\circ \med(a_1,\dots, a_n)=\med(Ta_1,\dots, Ta_n)$. An iteration of this equality yields that we can start with the threshold values if we are finally only interested in the thresholds.

\subsection{Neumann conditions}

In this section, we will consider how to adapt the scheme for a bounded domain with (in)homogeneous Neumann conditions added to the evolution. In the following, let $D$ denote an open Lipschitz bounded set in $\mathbb{R}^d$ (manifolds would work as well).
These Neumann conditions are in the homogeneous case the usual geometric assumption that for a characteristic function the (generalized) boundary of the set meets the domain boundary orthogonally. In accordance, for a level set function this means that all level sets fulfill this condition, i.e., $\nabla u\cdot \nu_D=0$ on $\partial D$. In the inhomogeneous case, we will prescribe a fixed contact angle $\alpha$ to the boundary, see Figure~\ref{Fig:Young angle}. This is called a Young angle condition and can be written as $\frac{\nabla u}{|\nabla u|}\cdot \nu_D=\cos\alpha$.

\begin{figure}[ht]
    \begin{tblr}{
      colspec = {X[c,h]},
      hlines = {0pt},
      vlines = {0pt},
      width=0.45\linewidth,
    }
    \includegraphics[width=\linewidth]{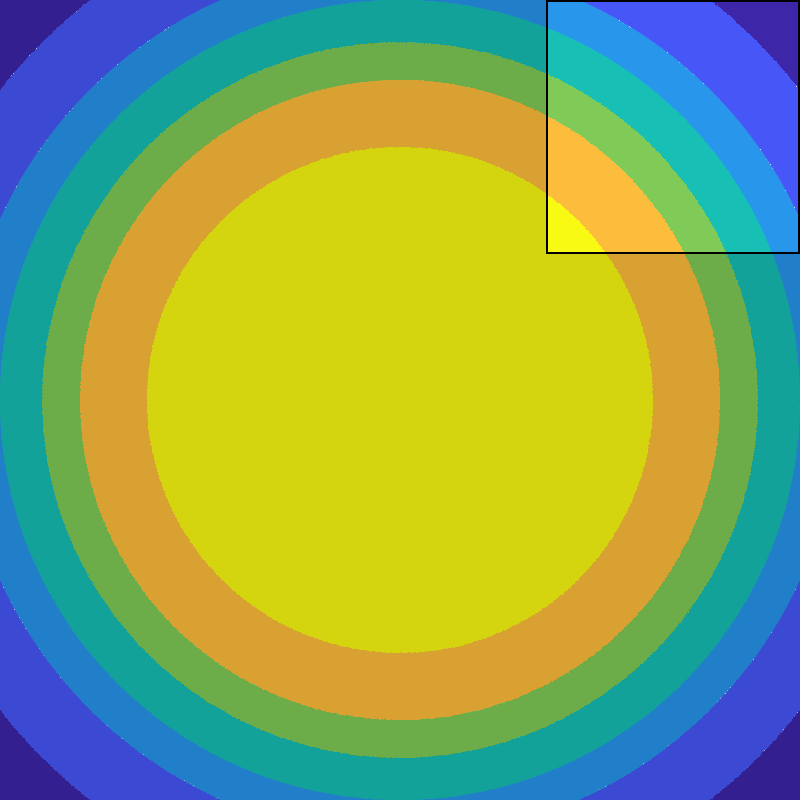}
    \end{tblr}
    \quad
    \begin{tblr}{
      colspec = {X[c,h]X[c,h]},
      hlines = {1pt},
      vlines = {1pt},
      width=0.45\linewidth,
    }
      \includegraphics[width=\linewidth]{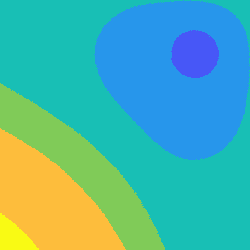} & \includegraphics[width=\linewidth]{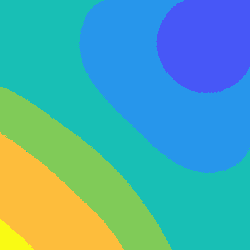} \\
      \includegraphics[width=\linewidth]{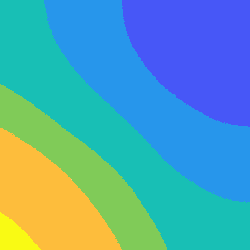} & \includegraphics[width=\linewidth]{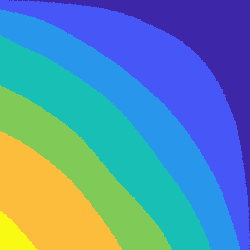} \\
    \end{tblr}
    \caption{Evolution of a quadratic radially symmetric function with different Young angle conditions. The prescribed angles are $0^\circ, 45^\circ, 90^\circ, 180^\circ$.}
    \label{Fig:Young angle}
\end{figure}

One important observation is that for the homogeneous case the Neumann conditions are incorporated in the thresholding scheme which also can be seen in the following way. We extend the characteristic function with the neutral value $\frac{1}{2}$ outside of the domain. This can be seen as mirroring the inside values and by symmetry, it yields the correct contact angle of $\frac{\pi}{2}$. The level set variant of this procedure is then that the update takes the restricted median to the inner of the domain, which can be seen as a mirroring as well.

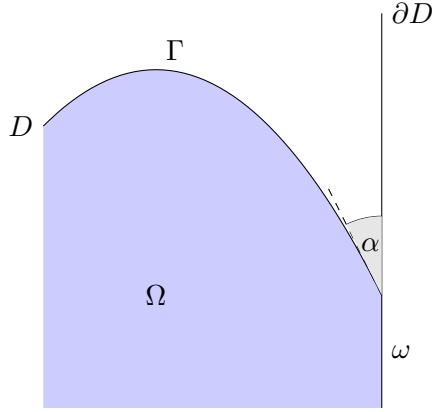
\begin{figure}[ht]
    \centering
        \begin{tikzpicture}[scale=1.5]
        \pgfsetlayers{pre main,main}
        \draw[domain=-1:2, smooth, variable=\x, black, name path = B, opacity=0] plot ({\x}, {-3});
        \draw[domain=-1:2, smooth, variable=\x, black, name path = A] plot ({\x}, {-0.5*\x*\x});
        \draw[dashed] (2,-2) -- (1.5,-1);
        \draw (2,-3) -- (2,0.5) node[right] {$\partial D$};
        \node[right] at (2,-2.5) {$\omega$};
        \node[left] at (-1,-0.5) {$D$};
        \node[above right] at (0,0) {$\Gamma$};
        \node at (0,-2) {$\Omega$};
        \draw (2,-2) ++(90:.7) arc (90:116.5:0.7);
        \fill[fill=gray!20!white] (2,-2) -- (2,-1.3) arc (90:116.5:0.7);
        \node at (1.9,-1.55) {$\alpha$};
        \tikzfillbetween[of=A and B]{blue, opacity=0.2};
        \end{tikzpicture}
    \caption{A visualization of the notation and Young angles.}
    \label{fig:YoungAngleVisualization}
\end{figure}

The energy one has to consider for a single set $\Omega\subseteq D$ is 
\begin{align*}
    E(\Omega)&\coloneqq\mathcal{H}^{d-1}(\partial^* \Omega\cap D)+\cos(\alpha)\mathcal{H}^{d-1}(\partial^*\Omega\cap \partial D)\\
    &=(1-\cos(\alpha))\mathcal{H}^{d-1}(\partial^* \Omega\cap D)+\cos(\alpha)\mathcal{H}^{d-1}(\partial^* \Omega).
\end{align*}
This energy is the same as in \cite{hensel2021bv}.
Equivalently, as $D$ is fixed, we can consider for any $c\in \mathbb{R}$ the energy
$$\mathcal{H}^{d-1}(\partial^* \Omega\cap D)+(c+\cos(\alpha))\mathcal{H}^{d-1}(\partial^*\Omega\cap \partial D)+c\mathcal{H}^{d-1}(\partial D\setminus \partial^*\Omega).$$
It can be seen that minimizers of this energy satisfy the Young angle condition by considering their Euler-Lagrange equation. Then, we can observe that by the gradient flow condition (the dissipation is exactly given by velocity and mean curvature) the contact angle at the boundary has to be constant and given by $\alpha$.

Heuristically this can be seen by the following formal computation for smooth sets. Let $\Gamma\coloneqq \partial \Omega\cap D$ be the inner free boundary and $\omega\coloneqq \partial^*\Omega\cap\partial D$ the shared boundary with the domain, see Figure~\ref{fig:YoungAngleVisualization}. Using \cite[Section 1.3]{ritore2023isoperimetric} and the divergence formula, we get for $\Omega_t\subseteq D$ a family of sets of finite perimeter with $B$ the motion vector field
\begin{align*}
     \frac{\dd}{\dd t}E(\Omega)&=\int\limits_{\Gamma}(\nabla \cdot B-\nu_\Omega\cdot\nabla B\nu_\Omega)\dd\mathcal{H}^{d-1}+\cos(\alpha)\int\limits_{\partial\omega}B\cdot\nu_{\omega}\dd\mathcal{H}^{d-2}.
\end{align*}
Using the distributional characterization of the mean curvature, this yields ($\overline{\Gamma}\cap \omega=\partial \omega=\partial\Gamma$)
\begin{align*}
     \frac{\dd}{\dd t}E(\Omega)&=\int\limits_{\Gamma} H(B\cdot \nu_\Omega)\dd\mathcal{H}^{d-1}+\int\limits_{\overline{\Gamma}\cap \omega}B\cdot \nu_{\Gamma}+\cos(\alpha) B\cdot \nu_{\omega}\dd\mathcal{H}^{d-2}.
\end{align*}

Observe, that $B$ always satisfies $B\cdot \nu_\Omega=V$ on $\Gamma$ and $\dot x=B$ for $x(t)\in \partial \omega$. Moreover, we can write $\nu_{\Gamma}$ as a linear combination of $\nu_D$ and $\nu_{\omega}$. Since we will not impose energy costs to lower dimensional sets, we can assume that the evolution is in an equilibrium at the domain boundary. Thus, the second integral on the RHS has to vanish for every admissible $B$ and since $\partial \omega\subseteq \partial D$, we get:
\begin{align*}
    0&=\int\limits_{\overline{\Gamma}\cap \omega}\nu_{\Gamma}\cdot \nu_{\omega} B\cdot \nu_{\omega}+\cos(\alpha) B\cdot \nu_{\omega}\dd\mathcal{H}^{d-2}.
\end{align*}
This implies a contact angle of $\alpha$ as we can choose test vector fields $B$:
\begin{align*}
    \nu_{\Gamma}\cdot \nu_{\omega} &= -\cos(\alpha) \text{ on }\overline{\Gamma}\cap \omega.
\end{align*}

The main observation in this section is the fact that we can adapt the MBO thresholding scheme to Neumann conditions by a simple extension. For characteristic functions this is done by considering a constant value outside of the domain. For the level set evolution, we use the layer cake approach to get the general scheme. In the following lemma, we state the precise formulation.

Similar to Definition~\ref{Def:AdmissibleKernel}, we call a kernel admissible for this scheme if the same assumptions are satisfied where the third assumption is replaced by
$\int\limits_{0}^\infty K(r)r^{d}\dd r<\infty.$
This is equivalent to the assumption that $k_1<\infty$, see Lemma~\ref{Lemma:NeumannConditions}.

\begin{defi}[Young angle condition --- scheme]
    \label{Def:YoungAngle}
    Let $D\subseteq \mathbb{R}^d$ be an open bounded Lipschitz domain and $K$ be an admissible kernel on $D$. Given a Young angle $\alpha\in (0,\pi)$, we define
    $$s\coloneqq\sin^2\left(\frac{\alpha}{2}\right)=\frac{1-\cos\alpha}{2}.$$
    Then, the Neumann MBO-thresholding scheme is defined on characteristic functions by
    $$\chi^{n+1}\coloneqq\mathds{1}_{\{K*(\chi^n\chi_D+s\chi_{D^c})\geq \frac{1}{2}\}}.$$
    Applied to level set functions, this yields the evolution 
    $$ u^{n+1}(x)\coloneqq \pmed\limits_{B_r(x)\cap D} u^n$$
    where $p$ is given by $$p=-(1-2s)\frac{|B_r(x)\cap D^c|}{|B_r(x)\cap D|}.$$
    
    Here, for $p\in [-1,1]$, the $p$-median $m=\pmed\limits_{A} u$ is defined as the infimum over all $m$ that satisfy
    $$p\leq \fint\limits_A \sign (m-u(y))\dd y.$$
    With $p=0$, one recovers the usual median.
    If $p>1$, $u$ attains the value $\infty$ and for $p<-1$, this becomes $-\infty$.
\end{defi}

Here, the notation $(K*f)(x)\coloneqq \int K(x,y)f(y)\dd y$ is a slight abuse of notation which becomes the usual convolution for $K(x,y)=K(x-y)$ and is used to show the similarity to the MBO scheme.
\begin{remark}
    In general, the minimization problems we analyze don't have unique minimizers. This is reflected in the schemes by the value $\chi$ takes at the boundary points where $$K*(\chi^n\chi_D+s\chi_{D^c})=\frac{1}{2}.$$
    For $u$, the choice lies in the fact that we take the infimum over all possible choices for $m$ which is coherent which the choice of $\geq$ in $\chi$ as they reflect the sub level sets of $u$. Other choices would be viable as well, this one is consistent with Section~\ref{DiscreteCase}.
\end{remark}
The step of how to get from the evolution for $\chi$ to the evolution of $u$ will be explained in the second part of the proof of Lemma~\ref{Lemma:NeumannConditions}.
\begin{remark}
    Other kernels and densities are also applicable as the calculations are all local. In the case of another kernel, we would get the defining equation (to be precise the infimum over the corresponding inequality):
    $$\int\limits_D K(x,y)\sign(u^{n+1}(x)-u^n(y))\dd y=-(1-2s)\int\limits_{D^c} K(x,y)\dd y.$$
\end{remark}

\begin{remark}
    The degenerate cases with $|p|>1$ can only occur for $s\neq \frac{1}{2}$ as with $s=\frac{1}{2}$, we have $p=0$ and are in the homogeneous Neumann case. Moreover, this case needs a specific density of our domain as can be seen in the discussion which follows the proof of the next lemma. 

    For the algorithm, we note that the $p$-median can be reformulated into a $k$-select criterion as it reflects that a percentage of values lies over our $p$-median. This also shows the connection to the $k$th-order statistic.
\end{remark}

The next lemma shows why we call this evolution the Neumann thresholding scheme.

\begin{lemma}[Neumann conditions - Young angle]
\label{Lemma:NeumannConditions}
    Let $\chi^n$ be a evolution of characteristic functions according to the thresholding scheme associated with the Young angle condition, see Definition~\ref{Def:YoungAngle}. 
    
    This evolution can be written as a minimizing movement scheme
    $$\chi^{n+1}\in\argmin\left\{E_h(\chi)+\frac{1}{2h}d_h^2(\chi, \chi^n)\right\}$$
    with
    \begin{align*}
        E_h&\coloneqq \frac{2s}{\sqrt{h}}\int\limits_D\int\limits_D K_h(x,y)|\chi(x)-\chi(y)|\dd y\dd x\\
        &\quad+\frac{1-2s}{\sqrt{h}}\int\limits_{\mathbb{R}^d}\int\limits_{\mathbb{R}^d} K_h(x,y)|\chi_D(x)\chi(x)-\chi_D(y)\chi(y)|\dd y\dd x,\\
        d_h^2(\chi, \chi^n)&\coloneqq 4\sqrt{h}\int\limits_{D}\int\limits_{D} K_h(x,y)(\chi(x)-\chi^n(x))(\chi(y)-\chi^n(y))\dd y\dd x.
    \end{align*}
    These energies $E_h$ $\Gamma$-converge to 
    $$k_1E(\Omega)= k_1\mathcal{H}^{d-1}(\partial^* \Omega\cap D)+k_1\cos(\alpha)\mathcal{H}^{d-1}(\partial^*\Omega\cap \partial D)$$ 
    in the strong $L^1$-topology. Moreover, they satisfy a $\Gamma$-compactness in the same topology.
\end{lemma}
Here, $k_1\coloneqq\int\limits_{\mathbb{R}^d}K(|x|)|x_1|\dd x$ depends only on the kernel used.
For the proof of this lemma, we will use $\Gamma$-convergence results of \cite{jonatim, TrillosSlepcev, esedoglu2015threshold}.

\begin{proof}
    We start with a reformulation of the scheme in a minimization scheme.
    For this, we consider characteristic functions with the evolution where we fix a value
    $$s=\frac{1-\cos\alpha}{2}=\sin^2\frac{\alpha}{2}$$
    outside of the domain. The evolution becomes:
    \begin{align*}
        \chi^{n+1}&\coloneqq\mathds{1}_{\{K*(\chi^n\chi_D+s\chi_{D^c})\geq \frac{1}{2}\}}\\
        &=\mathds{1}_{\{K*((2\chi^n-2s)\chi_D+2s-1)\geq 0\}}.
    \end{align*}
    This satisfies the following minimizing movement formulation
    \begin{align*}
        \chi^{n+1}&\in\argmin\limits_\chi\left\{\int\limits_{\mathbb{R}^d}\chi_D\chi K*((2s-2\chi^n)\chi_D+1-2s)\dd x\right\}\\
        &=\argmin\limits_\chi\left\{\int\limits_{\mathbb{R}^d}\chi_D\chi K*(2s\chi_D(1-\chi))\dd x\right.\\
        &\left.\hspace*{5em}+\int\limits_{\mathbb{R}^d}\chi_D\chi K*(2s\chi_D\chi-2\chi_D\chi^n+1-2s)\dd x\right\}\\
        &=\argmin\limits_\chi\left\{2s\int\limits_{\mathbb{R}^d}\chi_D\chi K*(\chi_D(1-\chi))\dd x\right.\\
        &\left.\hspace*{5em}+(1-2s)\int\limits_{\mathbb{R}^d}\chi_D\chi K*(1-\chi_D\chi)\dd x\right.\\
        &\left.\hspace*{5em}+\int\limits_{\mathbb{R}^d}\chi_D(\chi-\chi^n) K*(\chi_D(\chi-\chi^n))\dd x\right\}\\
        &=\argmin\limits_\chi\left\{2s\int\limits_{D}\int\limits_{D}K(x,y)|\chi(x)-\chi(y)|\dd x\right.\\
        &\left.\hspace*{5em}+(1-2s)\int\limits_{\mathbb{R}^d}\int\limits_{\mathbb{R}^d}K|\chi\chi_D-\chi\chi_D|\dd y\dd x\right.\\
        &\left.\hspace*{5em}+2\int\limits_{D}\int\limits_{D}K(x,y)(\chi-\chi^n)(x)(\chi-\chi^n)(y)\dd y\dd x\right\}.
    \end{align*}
    
    Now, we are able to show the $\Gamma$-converge of 
    $$2s\int\limits_{D}\int\limits_{D}K(x,y)|\chi(x)-\chi(y)|\dd x+(1-2s)\int\limits_{\mathbb{R}^d}\int\limits_{\mathbb{R}^d}K|\chi\chi_D-\chi\chi_D|\dd y\dd x$$
    to
    $$k_1E(\Omega)=2sk_1\mathcal{H}^{d-1}(\partial^*\Omega\cap D)+(1-2s)k_1\mathcal{H}^{d-1}(\partial^* \Omega).$$
    
    Indeed, this follows from \cite[Theorem 4.1]{TrillosSlepcev}. The first part of the energy $\Gamma$-converges for Lipschitz domains $D$ to the relative perimeter and fulfills a $\Gamma$-compactness for $s>0$. Both of these hold in the strong $L^1$-topology. We note that by symmetry (otherwise we can analyze $D\setminus \Omega$ instead of $\Omega$), we can assume $s\leq\frac{1}{2}$. The second part follows analogously to the first one if we restrict to sets contained in $D$. This is feasible due to the fact that the recovery sequence can be chosen to be the constant sequence and thus leaves $D$ unchanged. This fact also directly implies the existence of the recovery sequence for the sum. Lastly, we note that for $F_n\overset{\Gamma}{\to}F$ and $G_n\overset{\Gamma}{\to}G$, we have for all $x_n\to x:$
    $$F(x)+G(x)\leq \liminf\limits_{n\to\infty} F_n(x_n)+\liminf\limits_{n\to\infty} G_n(x_n)\leq \liminf\limits_{n\to\infty} F_n(x_n)+G_n(x_n).$$
    Therefore, we get the $\Gamma$-convergence and $\Gamma$-compactness for the whole energy and all $s\in [0,1]$. The proof extends to the setting where $D$ is an arbitrary bounded Lipschitz domain in $\mathbb{R}^d$ with a continuous density that is bounded from above and below by a positive constant.
    
    For the general case, we can use \cite[Theorem 3]{jonatim} building on \cite{esedoglu2015threshold}. The main observation is that we can see our energy as the energy of a three-phase system with phases $\Omega_1\coloneqq \Omega, \Omega_2\coloneqq D\setminus\Omega$ and $\Omega_3\coloneqq D^c$. Then,
    \begin{align*}
        E(\Omega)&=\mathcal{H}^{d-1}(\Sigma_{12})+(\cos(\alpha)+1)\mathcal{H}^{d-1}(\Sigma_{13})+\mathcal{H}^{d-1}(\Sigma_{23})
    \end{align*}
    where $\Sigma_{ij}$ denotes the interface between faces $\Omega_i$ and $\Omega_j$.
    
    Note, that these surface tensions are non-negative (and not all zero) and satisfy the triangle inequality. Then, it is feasible to use the result of \cite{jonatim} due to the fact that the recovery sequences are the constant sequences and thus, we don't leave our class of fixed $D$. 
    Using this notation, we can reformulate the minimization problem as 
    \begin{align*}
        \chi^{n+1}&\in\argmin\limits_\chi\left\{2s\int\limits_{\mathbb{R}^d}\int\limits_{\mathbb{R}^d}\chi_D\chi K(\chi_D(1-\chi))\dd y\dd x\right.\\
        &\left.\hspace*{5em}+(1-2s)\int\limits_{\mathbb{R}^d}\int\limits_{\mathbb{R}^d}\chi_D\chi K(1-\chi_D\chi)\dd y\dd x\right.\\
        &\left.\hspace*{5em}+\int\limits_{\mathbb{R}^d}\int\limits_{\mathbb{R}^d}\chi_DK(1-\chi_D)\dd y\dd x\right.\\
        &\left.\hspace*{5em}+\int\limits_{\mathbb{R}^d}\int\limits_{\mathbb{R}^d}\chi_D(\chi-\chi^n) K(\chi_D(\chi-\chi^n))\dd y\dd x\right\}\\
        &=\argmin\limits_\chi\left\{2s\int\limits_{\mathbb{R}^d}\int\limits_{\mathbb{R}^d}\chi_{\Omega_1} K \chi_{\Omega_2}\dd y\dd x\right.\\
        &\left.\hspace*{5em}+(1-2s)\int\limits_{\mathbb{R}^d}\int\limits_{\mathbb{R}^d}\chi_{\Omega_1}K (\chi_{\Omega_2}+\chi_{\Omega_3})\dd y\dd x\right.\\
        &\left.\hspace*{5em}+\int\limits_{\mathbb{R}^d}\int\limits_{\mathbb{R}^d}(\chi_{\Omega_1}+\chi_{\Omega_2})K \chi_{\Omega_3}\dd y\dd x\right.\\
        &\left.\hspace*{5em}+\int\limits_{\mathbb{R}^d}\int\limits_{\mathbb{R}^d}\chi_D(\chi-\chi^n) K(\chi_D(\chi-\chi^n))\dd y\dd x\right\}\\
        &=\argmin\limits_\chi\left\{\int\limits_{\mathbb{R}^d}\int\limits_{\mathbb{R}^d}\chi_{\Omega_1} K \chi_{\Omega_2}\dd y\dd x\right.\\
        &\left.\hspace*{5em}+(1+\cos\alpha)\int\limits_{\mathbb{R}^d}\int\limits_{\mathbb{R}^d}\chi_{\Omega_1}K \chi_{\Omega_3}\dd y\dd x\right.\\
        &\left.\hspace*{5em}+\int\limits_{\mathbb{R}^d}\int\limits_{\mathbb{R}^d}\chi_{\Omega_2}K \chi_{\Omega_3}\dd y\dd x\right.\\
        &\left.\hspace*{5em}+\int\limits_{\mathbb{R}^d}\int\limits_{\mathbb{R}^d}\chi_D(\chi-\chi^n) K(\chi_D(\chi-\chi^n))\dd y\dd x\right\}.
    \end{align*}
    This proof also works in the case where $D$ is merely a set of finite perimeter. The caveat of this method is that the kernels in \cite{esedoglu2015threshold} are more restrictive. To be precise, the conditions on $K$ are radial symmetry, integrability, non-negativity, $k_1<\infty$ as well as $|\nabla G(x)|\leq CG(\frac{x}{2})$ and $\nabla G(x)\cdot x\leq 0$. In \cite[Theorem 3]{jonatim}, only the heat kernel is considered. On the other hand, this work is done on smooth manifolds with smooth densities. An extension to more general domains can be achieved by Theorem~\ref{Thm:GammaTV}.

    Next, we prove how to get from the formulation for characteristic functions to the formulation for level set functions.

    For the level set function $u$, we will reformulate the minimizing movement into a form which can be integrated to yield a new evolution for $u^n$ which can be used to get the explicit form.
    \begin{align*}
        \chi^{n+1}&\in\argmin\limits_\chi\left\{\int\limits_{\mathbb{R}^d}\int\limits_{\mathbb{R}^d}\chi_D(x)\chi(x) K ((2s-2\chi^n(y))\chi_D(y)+1-2s)\dd y\dd x\right\}\\
        &=\argmin\limits_\chi\left\{\int\limits_{\mathbb{R}^d}\int\limits_{\mathbb{R}^d}2s K\chi_D(x)\chi_D(y)\chi(x)-2 K\chi_D(x)\chi_D(y)\chi(x)\chi^n(y)\right.\\
        &\hspace*{5em}+(1-2s)K\chi_D(x)\chi(x)\dd y\dd x\Bigg\}\\
        &=\argmin\limits_\chi\left\{\int\limits_{\mathbb{R}^d}\int\limits_{\mathbb{R}^d}K\chi_D(x)\chi_D(y)\chi(x)-2 K\chi_D(x)\chi_D(y)\chi(x)\chi^n(y)\right.\\
        &\hspace*{5em}+(1-2s)K\chi_D(x)\chi(x)(1-\chi_D(y))\dd y\dd x\Bigg\}\\
        &=\argmin\limits_\chi\left\{\int\limits_{\mathbb{R}^d}\int\limits_{\mathbb{R}^d}K\chi_D(x)\chi_D(y)|\chi(x)-\chi^n(y)|\right.\\
        &\hspace*{5em}+(1-2s)K\chi_D(x)\chi(x)(1-\chi_D(y))\dd y\dd x\Bigg\}.
    \end{align*}
    This can be used to compute the minimization problem for $u$ when plugging in $\chi_q^n=\mathds{1}_{\{u^n<q\}}$:
    \begin{align*}
        u^{n+1}&\in\argmin\limits_u\left\{\int\limits_{D}\int\limits_{D}K|u(x)-u^n(y)|\dd y\dd x+(1-2s)\int\limits_{D}u(x)\int\limits_{D^c}K\dd y\dd x\right\}.
    \end{align*}
    The Euler-Lagrange equation for almost every $x$ is precisely  
    \begin{align*}
        \int\limits_D K(x,y)\sign(u^{n+1}(x)-u^n(y))\dd y&=-(1-2s)\int\limits_{D^c}K(x,y)\dd y.
    \end{align*}
    Note, that it becomes two linked inequalities in the case of fattening as discussed in Section~\ref{Algorithm}.
    This amounts to the evolution
    $$u^{n+1}(x)\coloneqq {\pmed\limits_{D}}^K u^n$$
    with $p=-(1-2s)\int\limits_{D^c}K(x,y)\dd y$ which is for the normalized ball-kernel $p=-(1-2s)\frac{|B_r(x)\cap D^c|}{|B_r(x)\cap D|}$. As discussed previously, for $p>1$ we set $u^{n+1}=\infty$ and for $p=-1$, $u^{n+1}=-\infty$.
\end{proof}

For this scheme, with the ball kernel, it can happen that $p\not\in [-1,1]$ and thus $u$ would formally become the value $+\infty$ or $-\infty$. This can only happen if 
$$\frac{|B_r(x)\cap D^c|}{|B_r(x)\cap D|}>\frac{1}{|1-2s|}.$$
Equivalently in the ($r$-)density $\theta\coloneqq\frac{|B_r(x)\cap D|}{|B_r(x)|}$, this can be written as 
$$\theta<\frac{1}{1+\frac{1}{|1-2s|}}.$$
Thus, this phenomenon only occurs at corners and at high angles $\alpha$ in at most an $r$-neighborhood at each step.

To preserve the bounds on the value of the evolution, we propose a min-max-cutoff. For this, the value of $u$ is defined to be
$$ u^{n+1}(x)\coloneqq \pmed\limits_{B_r(x)\cap D} u^n$$
for $p\in (-1,1)$ and for $p\geq 1$, we define it as the limit of the $p$-median as $p\nearrow 1$. Similarly for $p\leq -1$. This corresponds to the local supremum or infimum, respectively. 
In the scheme, this can be recovered using a cutoff of the function $f(x)=-(1-2s)\frac{x}{1-x}$. So far, we had $p=f\left(\frac{|B_r(x)\cap D^c|}{|B_r(x)|}\right)$. This is replaced by $-1\lor f(x)\land 1$. In this case, the finite values of the evolution would remain the same whereas $\pm\infty$ would become the local supremum or infimum as desired. In the scheme for characteristic functions, this is reflected as a solution of the minimizing movement

\begin{align*}
    \chi^{n+1}&=\argmin\limits_\chi\left\{\left(2s\int\limits_D\int\limits_D K|\chi-\chi|\dd y\dd x +(1-2s)\int\limits_{\mathbb{R}^d}\int\limits_{\mathbb{R}^d}K|\chi\chi_D-\chi\chi_D|\dd y\dd x\right)\right.\\
    &\quad\land 2\left(\int\limits_D\int\limits_DK|\chi-\chi|\dd y\dd x+\int\limits_D\int\limits_D\chi K\chi\dd y\dd x\right)\lor \left(-2\int\limits_D\int\limits_D \chi K\chi \dd y\dd x\right)\\
    &\quad\left.+2\int\limits_D\int\limits_D (\chi-\chi^n)K(\chi-\chi^n)\dd y\dd x\right\}.
\end{align*}
We note that $\int\limits_D\int\limits_D\chi K\chi\dd y\dd x$ converges to the squared $L^2$-norm which is equivalent to the $L^1$-norm for characteristic functions. Moreover, if the kernel $K(x,y)$ has a positive Fourier transform and can be written as $K(x-y)$ (e.g., the heat kernel), it can be split and terms of the form $\int\int u K u\dd y\dd x$ can be seen as a $L^2$-norm $\norm{G*u}_{L^2}^2$ where $G$ is chosen s.t.\ $K=G*G$.

In detail, for the restricted functions $\chi_n\to \chi$ in $L^1(D)$, $\chi_n,\chi\in L^1(D;[0,1])$, we have
\begin{align*}
    \liminf\limits_{n\to \infty}\int \chi_n(x)\fint\limits_{B_{r(n)}(x)}\chi_n(y)\dd y\dd x>0 \text{ iff }\chi\not\equiv 0.
\end{align*}

First, for the kernel $K(x,y)=\frac{1}{\omega_d}\mathds{1}_{B_1(x)}(y)$, we compute
\begin{align*}
    \int\limits_D &\left|\frac{1}{r^d\omega_d}\int\limits_{D\cap B_r} \chi_n(y)\dd y-\chi(x)\right|\dd x\\
    &\leq \int\limits_D \frac{1}{r^d\omega_d}\int\limits_{D\cap B_r} \left|\chi_n(y)-\frac{r^d\omega_d}{|D\cap B_r|}\chi(x)\right|\dd y \dd x\\
    &\leq \int\limits_D \frac{1}{r^d\omega_d}\int\limits_{D\cap B_r} \left|\chi(y)-\frac{r^d\omega_d}{|D\cap B_r|}\chi(x)\right|\dd y \dd x\\
    &\quad+\int\limits_D \frac{1}{r^d\omega_d}\int\limits_{D\cap B_r}|\chi_n(y)-\chi(y)|\dd y \dd x\\
    &\leq \int\limits_D \frac{1}{r^d\omega_d}\int\limits_{D\cap B_r}\left|\chi(y)-\frac{r^d\omega_d}{|D\cap B_r|}\chi(x)\right|\dd y \dd x\\
    &\quad+\int\limits_D \frac{|D\cap B_r(x)|}{r^d\omega_d}|\chi_n(x)-\chi(x)|\dd x\\
    &\leq \int\limits_D \frac{1}{r^d\omega_d}\int\limits_{D\cap B_r}|\chi(y)-\chi(x)|\dd y \dd x+\int\limits_D \left|\frac{|D\cap B_r|}{r^d\omega_d}\chi(x)-\chi(x)\right| \dd x\\
    &\quad +\int\limits_D \frac{|D\cap B_r(x)|}{r^d\omega_d}|\chi_n(x)-\chi(x)|\dd x.
\end{align*}
This converges to 0 by the Lebesgue point theorem and dominated convergence.
Thus, we get

\begin{align*}
    &\left|\int\limits_D\chi_n(x)\frac{1}{r^d\omega_d}\int\limits_{D\cap B_r(x)}\chi_n(y)\dd y\dd x-\int\limits_D \chi^2(x)\dd x\right|\\
    &\quad\leq \int\limits_D|\chi_n(x)-\chi(x)|\frac{1}{r^d\omega_d}\int\limits_{B_r\cap D}|\chi_n(y)|\dd y\dd x\\
    &\quad\quad+\int\limits_D|\chi(x)| |\frac{1}{r^d\omega_d}\int\limits_{B_r\cap D}\chi_n(y)\dd y-\chi(x)|\dd x\\
    &\quad\leq \int\limits_D|\chi_n(x)-\chi(x)|\frac{|B_r\cap D|}{r^d\omega_d}\dd x+\int\limits_D \left|\frac{1}{r^d\omega_d}\int\limits_{B_r\cap D}\chi_n(y)\dd y-\chi(x)\right|\dd x.
\end{align*}
For a general admissible kernel, we note that since it is non-increasing and by continuity there are $c_1,c_2$ s.t.\ $K(x,y)\geq c_1\mathds{1}_{B_{c_2}(x)}(y)$. Thus, since $\chi_n,\chi\geq 0$, the result reduces to the standard case.

Hence, we can conclude that the bounds in the maximum and minimum of the energy go to $\pm\infty$ if $\chi\not\equiv 0$ after the normalization by the space size of the kernel. Therefore, we can conclude the $\Gamma$-convergence as the energy is $0$ for $\chi=0$.

The following theorem about the $\Gamma$-convergence of the non-local $\TV$-energy is a simple extension of \cite[Theorem 4.1]{TrillosSlepcev}. It follows analogously to the convergence of the Dirichlet energies shown in \cite[Theorem 12, 14]{jonatim}.

\begin{thm}[Convergence of the $\TV$-energy]
\label{Thm:GammaTV}
Let $K$ be an admissible kernel on a compact manifold $M$ (or $\mathbb{R}^d$) with a continuous density $\rho>0$ and smooth boundary. Additionally, let $D$ be a relatively open bounded Lipschitz set in the inner of $M$.  
Let 
$$E_h(u)\coloneqq \frac{1}{\sqrt{h}}\int\limits_D\int\limits_D K_{\sqrt{h}}(d_M(x,y))|u(x)-u(y)|\rho(x)\rho(y)\dd \vol(y)\dd \vol(x)$$
for $u\in L^1(M;[0,1])$ and define the total variation
$$E(u)\coloneqq \sup\left\{\int\limits_D u{\ddiv}_M \varphi \dd \vol(x):|\varphi(x)|\leq \rho^2(x), \varphi\in \Gamma(TM)\right\}.$$
Then, $E_h$ $\Gamma$-converges to $k_1 E$ in the strong $L^1(D)$-topology. Moreover, they satisfy the $\Gamma$-compactness in the same topology.
\end{thm}
For smooth $u$, the total variation can be written as $E(u)=\int\limits_D |Du|\rho^2\dd \vol$. For characteristic functions, we call $E(\chi)$ also the perimeter. This theorem can be used analogously to \cite{TrillosSlepcev} and \cite{jonatim} to get a $\TL^1$-$\Gamma$-convergence result of the sampled discrete energies.

In the same fashion as before, we extend Lemma~\ref{Lemma:NeumannConditions} to this setting.

\subsection{Dirichlet conditions and SSL}

The problem of Dirichlet conditions and semi-supervised learning are related in the way that they partially impose fixed values on our solution. For Dirichlet conditions, these values are given on a larger area which usually has positive capacity and only have to be satisfied in the trace sense or a similar way. Thus, it suffices to extend the domain in a $r$ wide range and fix values according to (an extension of) the given boundary values on this domain. Then, the algorithm will obey this fixed values and attain them at the boundary. It is noteworthy that even in this simpler case the formulation of the evolution with Dirichlet values is not trivial, cf.~\cite{bian2023level} where the level set mean curvature flow with Dirichlet boundary conditions is defined via an obstacle problem.

For semi-supervised learning, the problem is more complicated as mostly only single values have prescribed values. In the limit, these values would vanish and thus, one has to impose more restrictions to retain their values. Usually, one is interested in a weak form of continuity. There are different possibilities to handle this. One possibility would be to fix a large enough neighborhood and prescribe the same value to every sampled point in this neighborhood. This would fall in the vicinity of Dirichlet boundary values. Another way is to dictate a local profile which the solution has to follow.

Often one of two different approaches is taken. For one, we can change the kernel to take into account the given values. In this way, we impose a weight that places a positive capacity onto these single values. Specifically, we follow~\cite{CalderSlepcev} who analyze SSL on minimizers of the Dirichlet energy. They impose an additional kernel factor $\gamma(x,y)$ which depends on the vicinity to the prescribed values (which we call $\Gamma$):
\begin{align*}
    \gamma(x,y)&\coloneqq \frac{\gamma_\zeta(x)+\gamma_\zeta(y)}{2},\\
    \gamma_\zeta(x)&\coloneqq \zeta\land \left(1+\left(\frac{r_0}{\dist(x,\Gamma)}\right)^\alpha\right)\text{ for }\dist(x,\Gamma)\leq R.
\end{align*}
Here, $\zeta$ is a cutoff to avoid singularities and will go to $\infty$ sufficiently fast, $r_0, R$ are length parameters to tweak the effect of the weight and $\alpha$ determines the strength and needs to be larger than $d-2$. Additionally, $\gamma$ is extended outside the neighborhood of size $R$. Larger $\alpha$ and smoother extensions lead to smoother evolutions as can be seen in \cite{CalderSlepcev}. We note, that the choice of $\gamma$ dictates a local profile.

Since all previous computations carry over, we obtain the evolution starting with $u^0\coloneqq g$:
$$u^{n+1}\in \argmin\left\{\int\limits_D\int\limits_D \gamma K(x,y)|u(x)-u^n(y)|\dd y\dd x\right\}.$$

Similarly, \cite{MBOSSLBertozzi} changes the differential operator with a fidelity term which would also result in a different kernel. This idea is related to directly adding a fidelity or compliance term to the energy as done in \cite{jacobs2017fast, jonatim}. In this case, we add an additional force term that enforces the Dirichlet values. The idea is to precompute a function with the $\Delta^\infty$-operator that gives a direction and strength depending on the distance to the Dirichlet values. This procedure cannot directly be carried over to the level set functions as the precomputed functions contains other values than $0$ and $1$. In the case where the function is discrete, we would obtain a strong enforcement with the $\pmed$ where $p=\sign(u-\text{Dirichlet value})$ on $\Gamma$. Instead, we can change this and take the $\pmed$ with $p=\gamma \cdot v_1\sign(u-v_2)$ where $v_2$ is a function extending the Dirichlet values in a smooth way and $v_1\in [0,1]$ determines the strength. Additionally, $\gamma$ is a scaling factor.
The functions $v_1, v_2$ can be computed by the discrete version of $\Delta^\infty$. In the graph case with a kernel that is a characteristic function, the defining equation becomes:
\begin{align*}
    \begin{cases}
        v_1(x)=1 &\text{on }\Gamma,\\
        v_1(x)=\frac{\max\limits_A v_1(y) +\min\limits_A v_1(y)}{2} &\text{else}.
    \end{cases}
\end{align*}

\section{Related Work}
\label{History}

The analysis of mean curvature flow has a very extensive history. We will concentrate on the more recent development of approximation schemes for mean curvature flow. Of particular note is the MBO scheme also called the thresholding scheme due to \cite{merriman1992diffusion, merriman1994motion}. It introduces the idea of a two-step scheme to approximate mean curvature flow in an effective and simple way as described in Section~\ref{Algorithm}. 
It is also possible to talk about this scheme after the onset of singularities. These can develop in finite time due to the degenerate nature of mean curvature flow, even for smooth initial data. This raises the question of weak solution concepts. A particularly useful one for level set mean curvature flow is that of viscosity solutions. It has been introduced separately by \cite{ChenGigaGoto} and \cite{ESI}. In the series of papers of \cite{ESI} as well as in \cite{laux2023genericlevelsetsmean, hensel2021newvarifoldsolutionconcept}, the connection between different weak (and strong) solution concepts is made. The convergence of the thresholding scheme to the viscosity solution of level set mean curvature flow was proved in \cite{evans1993convergence}.
Another efficient algorithm for the case of a single characteristic set until the onset of fattening or the development of singularities is given in \cite{chambolle2004algorithm} based on \cite{AlmgrenTaylorWang}. This scheme has also been analyzed in \cite{dephilippis2019implicittimediscretizationmean}.
The idea of median filters with the observation that the level set approach of the thresholding scheme has the same minimization problem of medians was introduced in \cite{Oberman}. Further generalizations were made in \cite{esedoglumedianfilter}. An important quantity in the analysis of these schemes is the associated energy of the thresholding scheme, the heat content. The convergence of this energy in the multiphase setting was proved in \cite{esedoglu2015threshold, laux2019thresholdingschememeancurvature}.
The first approach is based on a monotonicity in time while the other one is based on the proof of \cite{alberti1998non}.
We additionally consider a fully discretized scheme with a point process for the space discretization. This perspective of space discretizations in addition to the non-locality was taken in \cite{CalderSlepcev,TrillosSlepcev, SlepcevTrillos2} and extended in \cite{jonatim}. Similarly, \cite{van2014mean} also analyzed the MBO scheme on discrete graphs.
A similar approach for the fixed grid in 2d with a 5-point stencil is analyzed in \cite{misiatsYip}. The main question of this work is which regimes lead to convergence and when pinning or freezing occurs.

More specifically, in \cite{CalderSlepcev}, it is shown that the Dirichlet energies converge $E_{r,N}\overset{\Gamma}{\to}k_2E$ in ${\TL}^2(D)$ in the setting of euclidean space without density as long as $r\gg d_\infty(\mu,\mu_N)$. 
The work \cite[Theorem 1.4]{SlepcevTrillos2} improves this to optimal setting where $r\gg \frac{\log^\frac{1}{d}N}{N^\frac{1}{d}}$ is bounded by the connectivity radius. This is done for Lipschitz domains $D$ with continuous density and a general kernel as described in the discussion at the end of Section~\ref{TL-convergence}. Building on this, \cite{jonatim} extends the results to the setting of smoothly weighted manifold.

\section{Outlook}
\label{Outlook}

A possible direction which seems promising in proving the convergence of the discrete MBO scheme is to lift Section~\ref{TL-convergence} to the $L^1-$setting. For this, one needs general ${\TL}^p-$spaces introduced in \cite{TrillosSlepcev}.

In the following, let $1\leq p<\infty$.

\begin{defi}[${\TL}^p$-space]
    We define the space of compatible measure-function pairs $(\mu,f)$ as 
    $${\TL}^p(X)\coloneqq \{(\mu,f) \left| \mu\in \mathcal{P}(X), f\in L^p(X,\mu) \right\}.$$
\end{defi}

\begin{defi}[${\TL}^p$-convergence]
    We define for $(\mu,f),(\nu,g)\in {\TL}^p(X)$ the distance $d_{{\TL}^p}$ as
    $$ d_{{\TL}^p}((\mu,f),(\nu,g))\coloneqq \inf\limits_{\pi\in \Gamma(\mu,\nu)}\left(\int\limits_X\int\limits_X |x-y|^p+|f(x)-g(y)|^p\dd\pi(x,y)\right)^\frac{1}{p}$$
    and say that a sequence $\{(\mu_n,f_n)\}_{n\in\mathbb{N}}$ converges in ${\TL}^p$ if
    $$d_{{\TL}^p}((\mu_n,f_n),(\mu,f))\to 0.$$
    Here, $\Gamma(\mu,\nu)$ is the set of all transport plans from $\mu$ to $\nu$.
\end{defi}

There are a few difficulties with this approach. For one, we used the implicit Euler-scheme and one would want an explicit scheme to achieve a median filter. In this case, the distance would become more degenerate, especially if the kernel used is not positive everywhere under the Fourier transform.

For instance, in the heat equation, the explicit Euler scheme with evolution
\begin{align*}
    u_N^{n+1}&\coloneqq \argmin\left\{\sum\limits_x\frac{1}{\#(B_r(x))}\sum\limits_{y\in B_r(x)}|u(x)-u_N^n(y)|^2\right\}\\
    &=\frac{1}{\#(B_r(x))}\sum\limits_{y\in B_r(x)}u_N^n(y),\\
    \frac{u_N^{n+1}(x)-u_N^n(x)}{r^2}&=\frac{\frac{1}{\#(B_r(x))}\sum\limits_{y\in B_r(x)}(u_N^n(y)-u_N^n(x))}{r^2}.
\end{align*}
would lead to the distance
$$d_N(u,u_N^n)=\int(u(x)-u_N^n(x))\fint_{B_r(x)}(u(y)-u_N^n(y))\dd\mu_N\dd\mu_N.$$
This distance still is a continuous perturbation (using the $\infty-$transport plans) but not necessarily non-negative and does not need to satisfy the triangle inequality. It is the same distance as for the minimizing movement formulation of the MBO scheme. 
For kernels other than the ball (any kernel with a non-negative Fourier transform), this would still make sense and lead to convergence.

Another problem arises directly in the $L^1$-formulation of $\TV$. This minimizing movement would look more like a rate independent system with an inhomogeneous scaling. In addition, the energy would no longer be uniformly convex.

A different interesting question is that of anisotropy. Similar to \cite{IPS}, one could formulate general curvature dependent motions which do need not to be isotropic. In these cases, all our questions could be asked as well.

An additional interesting direction is to change the algorithm from a Jacobi formulation to a Gauss-Seidel form. I.e., one would take random points and update them according to the median on the constructed graph. This version yields similar results in simulations.

\section{Acknowledgment}

The authors would like to thank Prof.\ Felix Otto for many fruitful discussions and help in this work.

\frenchspacing
\bibliographystyle{abbrv}
\bibliography{References.bib} 
\end{document}